\documentclass[11pt]{article}
\usepackage[utf8]{inputenc} 
\usepackage[T1]{fontenc} 
\usepackage{lmodern}  
\usepackage[a4paper]{geometry} 
\usepackage{amsmath, amssymb} 
\usepackage{amsthm} 
\usepackage{graphicx} 
\usepackage{xcolor} 
\usepackage{pgf,tikz}
\usetikzlibrary{arrows}
\usepackage[perpage]{footmisc} 
\usepackage{hyperref} 
\usepackage{array} 
\usepackage{url} 
\usepackage{verbatim}
\usepackage{multicol}
\usepackage{mdframed}
\usepackage{pstricks}
\usepackage{pstricks-add}
\usepackage{tabularx}
\usepackage{pgf,tikz}
\usetikzlibrary{arrows}
\usepackage{fancybox}
\usepackage{stmaryrd}
\usepackage{comment}

\usepackage{minitoc}
\mtcsetfeature{minitoc}{after}{\vspace{24pt}}

\usepackage{enumitem}
\setlist[itemize]{label=\textbullet}

\theoremstyle{remark}
\newtheorem*{rmq}{Remark}
\newtheorem*{ex}{Example}
\newtheorem*{notation}{Notation}
\theoremstyle{plain}
\newtheorem*{rappelprop}{Proposition}
\newtheorem*{rappel}{Theorem}

\theoremstyle{plain}
\usepackage{thmbox}

\newtheorem{theorem}{Theorem}[section]
\newtheorem{prop}[theorem]{Proposition}
\newtheorem{lemma}[theorem]{Lemma}
\newtheorem{cor}[theorem]{Corollary}
\newtheorem{definition}[theorem]{Definition}

\usepackage[all]{xy} 

\newcommand\blfootnote[1]{%
  \begingroup
  \renewcommand\thefootnote{}\footnote{#1}%
  \addtocounter{footnote}{-1}%
  \endgroup
}

\usepackage{fancyhdr}
\DeclareMathOperator{\spec}{Spec}
\DeclareMathOperator{\spm}{Spm}
\DeclareMathOperator{\fracloc}{Frac}
\DeclareMathOperator{\sing}{Sing}
\DeclareMathOperator{\im}{Im}

\DeclareMathOperator{\Norm}{Norm}
\DeclareMathOperator{\SN}{SN}
\DeclareMathOperator{\disc}{disc}
\DeclareMathOperator{\Cond}{Cond}

\newcommand{\m}{\mathfrak{m}}

\newcommand{\p}{\mathfrak{p}}
\newcommand{\q}{\mathfrak{q}}

\newcommand{\Ocal}{\mathcal{O}}
\newcommand{\Ccal}{\mathcal{C}}
\newcommand{\Zcal}{\mathcal{Z}}
\newcommand{\Jcal}{\mathcal{J}}
\newcommand{\Vcal}{\mathcal{V}}
\newcommand{\Dcal}{\mathcal{D}}

\newcommand{\Rrm}{\mathrm{Rad}}
\newcommand{\inj}{\hookrightarrow}

\newcommand{\KO}{\mathcal{K}^0}
\newcommand{\K}{\mathcal{K}}
\newcommand{\piC}{\pi_{_\mathbb{C}}}
\newcommand{\pik}{\pi_k}
\newcommand{\C}{\mathbb{C}}
\newcommand{\R}{\mathbb{R}}
\newcommand{\A}{\mathbb{A}}

\makeatletter
\renewcommand\tableofcontents{%
  \null\hfill\textbf{\Large\contentsname}\hfill\null\par
  \@mkboth{\MakeUppercase\contentsname}{\MakeUppercase\contentsname}%
  \@starttoc{toc}%
}
\makeatother

\author{FRANÇOIS BERNARD}
\title{\textbf{Seminormalization and regulous functions on complex affine varieties}}
\date{}
\begin{document}

\maketitle
\vspace{-0.6cm}
\begin{abstract}\noindent
   We study seminormalization of affine complex varieties. We show that polynomials on the seminormalization correspond to the rational functions which are continuous for the Euclidean topology. We further study this type of functions which can be seen as complex \textit{regulous} functions, a class of functions recently introduced in real algebraic geometry, or as the algebraic counterpart of \textit{c-holomorphic} functions.
\end{abstract}

\makeatletter
\let\Hy@linktoc\Hy@linktoc@none
\makeatother

{\setlength{\baselineskip}{0.1\baselineskip}
\tableofcontents}

\makeatletter
\def\blfootnote{\gdef\@thefnmark{}\@footnotetext}
\makeatother


\vspace{-0.3cm}
\section{Introduction.}

The present paper is devoted to the study of seminormalization of affine complex varieties, to its link with continuous rational functions and to the study of those functions. The operation of seminormalization was formally introduced around fifty years ago in the case of analytic spaces by Andreotti and Norguet \cite{Andre}. For algebraic varieties, the seminormalization $X^+$ of $X$ is the biggest intermediate variety between $X$ and its normalization which is bijective with $X$. Recently, the concept of seminormalization appears in the study of singularities of algebraic varieties, in particular in the minimal model program of Kollár and Kovács (see \cite{Kollar2} and \cite{Kollar3}). The seminormalization has the property to have "multicross" singularities in codimension 1 (see \cite{Multicross}), it means that they are locally analytically isomorphic to the union of linear subspaces of affine space meeting transversally along a common linear subspace.\\
Around 1970 Traverso \cite{T} introduced the notion of the seminormalization $A^+_B$ of a commutative ring $A$ in an integral extension $B$. The idea is to glue together the prime ideals of $B$ lying over the same prime ideal of $A$. The seminormalization $A^+_B$ has the property that it is the biggest extension $C$ of $A$ in $B$ which is subintegral i.e. such that the map $\spec(C)\to\spec(A)$ is bijective and equiresidual (it gives isomorphisms between the residue fields). We refer to Vitulli \cite{V} for a survey on seminormality for commutative rings and algebraic varieties. See also \cite{LV}, \cite{V2} and \cite{Swan} for more detailed informations on seminormalization. \\
In the paper \cite{LV}, the authors tried to identify the coordinate ring of the seminormalization of a variety as the ring of rational functions which are continuous for the Zariski topology. Unfortunately, the Zariski topology is not strong enough for this to be true. The first aim of this paper is to show that the correct functions to consider are rational functions which are continuous for the Euclidean topology. The idea of studying the concept of seminormalization with that kind of functions comes from \cite{FMQ} and \cite{Central} in the context of real algebraic geometry. Those functions appeared recently in real algebraic geometry (see \cite{FHMM} and \cite{Kollar}) under the name of "regulous functions". They allow to recover some classical theorems of complex algebraic geometry, such as the Nullstellensatz, which normally do not hold anymore in real algebraic geometry. A complex analog of regulous functions has been studied in \cite{BDTW} and \cite{BDT} in the point of view of complex analytic geometry. The second aim of this paper is to bring a study of complex regulous functions in the point of view of complex algebraic geometry.

The paper is organized as follows. In Section \ref{SectionUniversalProp} we rewrite Traverso's construction of the seminormalization of a ring and its universal property regarding to subintegral extensions of rings.\\
In Section \ref{SectionUniversalPropGeo} we look at the seminormalization of an affine variety over an algebraically closed field of characteristic zero and to its universal property. The seminormalization of an affine variety $X$ can be seen as the biggest birational variety such that its closed points are in bijection with those of $X$.\\
In Section \ref{SectionContinuousRational} we introduce continuous rational functions on a complex affine variety $X$. More precisely we consider the functions $f : X(\C) \to \C$ which are rational on a Zariski dense open set of $X(\C)$ and which are continuous, for the Euclidean topology, on all $X(\C)$. The ring of those functions is denoted by $\KO(X(\C))$. For $\pi : Y \to X$ a finite morphism between complex affine varieties and $\piC$ its restriction to $Y(\C)$, we look at the induced morphism $f\mapsto f\circ \piC$ between rings of continuous rational functions. We show that the image of such a morphism is 
$$ \{ f\in \KO(Y(\C))\text{ with $f$ constant on the fibers of }\piC \}$$
It allows us to reinterpret subintegral extensions between coordinate rings of varieties. For $X$ and $Y$ two varieties, one gets that $\C[X] \inj \C[Y]$ is subintegral if and only if $\KO(X(\C)) \simeq \KO(Y(\C))$. The first half of this paper can be summarized by the following result :

\begin{rappel}[\ref{TheoSubSsiBij} and \ref{PropSubEquiMemeFctRatioCont}]
Let $\pi : Y \to X$ be a finite morphism between affine complex varieties. Then the following properties are equivalent :
\begin{enumerate}[topsep=0pt, partopsep=0pt, itemsep=0pt,parsep=0pt]
    \item[1)] $\pi$ is subintegral.
    \item[2)] $\piC$ is bijective.
    \item[3)] The rings $\KO(Y(\C))$ and $\KO(X(\C))$ are isomorphic.
    \item[4)] $\piC$ is an homeomorphism for the Euclidean topology.
    \item[5)] $\piC$ is an homeomorphism for the Zariski topology.
    \item[6)] $\pi$ is an homeomorphism.
\end{enumerate}
\end{rappel}
In the beginning of Subsection \ref{ConnectionBetweenRat}, we prove that continuous rational functions are regular on the smooth points of a variety. It allows us to see that, for $X$ a normal variety, the thinness of $\sing(X)$ implies $\KO(X(\C)) = \C[X]$. This fact combined with the previous theorem leads us to the main result of this paper :
\begin{rappel}[\ref{TheoRatioContEgalPolySN}]
Let $X$ be an affine complex variety and $\pi^+ : X^+ \to X$ be the seminormalization morphism. We have the following isomorphism
$$\begin{array}{lccc}
     \varphi : & \KO(X(\C)) & \xrightarrow{\sim} & \C[X^+]\\
     & f & \mapsto & f\circ \piC^+ 
\end{array}$$
\end{rappel}
The results of Subsection \ref{ConnectionBetweenRat} can be summarized with the following diagramm: for every morphism $\pi : Y\to X$ such that $\C[X]\inj \C[Y]$ is subintegral, we get
$$\xymatrix{
    \KO(X(\C)) \ar[r]^{\simeq}& \KO(Y(\C)) \ar[r]^{\simeq} & \KO(X^+(\C)) \ar@{^{(}->}[r] & \KO(X'(\C))\\
    \C[X] \ar@{^{(}->}[u] \ar@{^{(}->}[r]^{subint.} & \C[Y] \ar@{^{(}->}[u] \ar@{^{(}->}[r]^{subint.} & \C[X^+] \ar@{=}[u] \ar@{^{(}->}[r] & \C[X'] \ar@{=}[u]
}$$
In subsection \ref{SubsectionContinuousRationalvsRegulous} we look at a consequence of theorem \ref{TheoRatioContEgalPolySN} : the restriction of a complex continuous rational function on a subvariety is still a rational function. It is an interesting fact because it says that, unlike the real case, continuous rational functions are regulous functions.\\
In the remaining of Section \ref{SectionContinuousRational}, we are interesting in finding criteria for a continuous function to be rational and then for a rational function to be continuous. In Subsection \ref{TheRingOfRat}, we show that a continuous function on $X(\C)$ which is a root for a polynomial with coefficients in $\C[X]$ is necessarily rational. It implies two results. First, we get that $\KO(X(\C))$ is the integral closure of $\C[X]$ in $\Ccal^0(X(\C),\C)$. 
Secondly, we get an algebraic version of Whitney's theorem 4.5Q in \cite{Whitney}, saying that a continuous function on the closed points of an affine variety is rational if and only if its graph is Zariski closed. The second point says that c-holomorphic functions with algebraic graph studied in \cite{BDTW} and \cite{BDT} correspond, for algebraic varieties, to the continuous rational functions considered in this paper.\\
Finally, in Subsection \ref{ExamplesOfRat}, we give several nontrivial examples of continuous rational functions thanks to the following criterion. 
\begin{rappel}[\ref{TheoGraphEntiereDoncCont}]
Let $X$ be an affine complex variety. A function $f : X(\C) \to \C$ is a continuous rational function if and only if it is rational, integral over $\C[X]$ and its graph is Zariski closed.
\end{rappel}
In section \ref{SectionClassicalResults} we reinterpret several classical results about seminormalization in terms of rational continuous functions. In the first subsection, we look at criteria for a variety to be seminormal given by Leahy-Vitulli, Hamann and Swan (see the review \cite{V}). To prove that those criteria are sufficient, we show that if $f$ is an element of $\KO(X(\C))\setminus \C[X]$, then we can always find a function $g\in\C[X][f]\setminus\C[X]$ such that $g^n\in\C[X]$ for all $n\geqslant 2$. To see that they are necessary, we construct explicit continuous rational functions from the relations appearing in the different criteria. The second subsection is dedicated to see what the commutation between the localization and the seminormalization means for continuous rational functions.\\
In section \ref{SectionTheSheafOf} we define the sheaf $\KO_X$ of complex regulous functions and we generalize the main result of this paper by showing that, for an affine variety $X$, the ringed space $(X, \KO_X)$ is isomorphic to the affine scheme $(X^+,\Ocal_{X^+})$. A generalization for general algebraic varieties over a field of characteristic 0 can be found in the forthcoming paper \cite{BFMQ}.

\textbf{Acknowledgement :} This paper is part of Ph.D. Thesis of the author. The author is deeply grateful to G. Fichou and J.-P. Monnier for their precious help.

\section{Universal property of the seminormalization.}\label{SectionUniversalProp}
\vspace{-0.1cm}
We recall in this section the construction of the seminormalization introduced by Traverso \cite{T} for commutative rings. This construction is linked to the notion of subintegrality in the way that the seminormalization of a ring is its biggest subintegral extension.

Let $A$ be a ring, we note $\spec(A) := \{\p \subset A \mid \p \text{ is a prime ideal of }A\}$ the spectrum of $A$ and $\spm(A) := \{\m \subset A \mid \m \text{ is a maximal ideal of }A\}$ the maximal spectrum of $A$. Let $\p \in \spec(A)$, then $A_{\p} := (A\setminus \p)^{-1}A$ is the localization of $A$ at $\p$ and $\kappa(\p) := A_{\p}/\p A_{\p}$ the residue field of $\p$.

Since the seminormalization is defined for integral extensions, we recall this notion here.

\begin{definition}
Let $A\inj B$ be an extension of rings.
\begin{enumerate}
    \item An element $b\in B$ is \textit{integral} over $A$ if there exists a monic polynomial $P\in A[X]$ such that $P(b)=0$.
    \item We call \textit{integral closure} of $A$ in $B$ and we write $A'_B$ the ring defined by $$A'_B := \{ b\in B \mid b \text{ integral on }A\}$$
    \item \vspace{-0.2cm}The extension $A\inj B$ is integral if $A'_B = B$.
\end{enumerate}\vspace{-0.4cm}
\end{definition}
\vspace{-0.1cm}
Now we define the seminormalization of a ring in an integral extension. The idea behind this definition is to \textit{glue} the prime ideals of $B$ above those of $A$. If one thinks of it in terms of algebraic varieties, it consists of gluing points in the fibers together.

\begin{definition}
Let $A \inj B$ be an integral extension of rings. We define

$$A_B^{+} := \{b \in B \mid \forall \p \in \spec(A) \text{, } b_{\p} \in A_{\p}+\Rrm(B_{\p})\}$$

where $\Rrm(B_{\p}) :=  \displaystyle{\bigcap_{\m\in\spm(B_{\p})}\m }$ is the Jacobson radical of $B_{\p}$.

We say that $A_B^+$ is the \textit{seminormalization} of $A$ in $B$. If $A = A_B^+$, then $A$ is said to be \textit{seminormal} in $B$.
\end{definition}

We introduce now the notion of subintegral extension which is strongly related with that of seminormalization.

\begin{definition}
An integral extension of rings $A \inj B$ is called \textit{subintegral} if the two following conditions hold :
\begin{enumerate}
    \item The induced map $\spec(B) \to \spec(A)$ is bijective.
    \item For all $\p\in \spec(A)$ and $\q \in \spec(B)$ with $\q\cap A = \p$, the induced map on the residue fields $\kappa(\p) \inj \kappa(\q)$ is an isomorphism.
\end{enumerate}

When the second condition holds, we say that $A\inj B$ is \textit{equiresidual}.
\end{definition}

The following statement gives the link between the two last definitons. It gives us a universal property of the seminormalization : the seminormalization of a ring in another one as its biggest subintegral subextension.

\vspace{-0.5cm}
\begin{prop}\label{PropPU}
Let $A\inj C\inj B$ be integral extensions of rings. Then the following statements are equivalent :
\begin{enumerate}
    \item[1)] The extension $A \inj C$ is subintegral.
    \item[2)] The image of $C\inj B$ is a subring of $A^+_B$.
\end{enumerate}
\end{prop}

Before proving this, we give several lemmas. First, we recall some classical algebraic results. Those are crucial to make things work and they are the reason we consider integral extensions.

\begin{lemma}[\cite{AM}]\label{LemGoingUp}
Let $A\inj B$ be an integral extension of rings. Then
\begin{enumerate}
    \item (\textit{Lying-Over} property) The map $\spec(B) \to \spec(A)$ is surjective.
    \item Let $\p\in \spec(A)$ and $\q\in \spec(B)$ such that $\q\cap A = \p$. Then $\q\in \spm(B)$ if and only if $\p\in\spm(A)$. 
    \item (\textit{Going-up} property) Let $\p \subseteq \p'\in \spec(A)$ and $\q \in \spec(B)$ such that $\q\cap A = \p$. Then there exists $\q \subseteq \q'\in \spec(A)$ such that $\q' \cap A = \p'$.
    \item Let $S$ be any multiplicative subset of $A$, then $S^{-1}A \inj S^{-1}B$ is integral.
\end{enumerate}
\end{lemma}

Now see that a subextension of a subintegral extension is necessarily subintegral.
\begin{lemma}\label{LemTransSub}
Let  $A \inj C \inj B$ be integral extensions of rings. Then the following properties are equivalent 
\begin{enumerate}
    \item[1)] The extension $A\inj B$ is subintegral.
    \item[2)] The extensions $A\inj C$ and $C\inj B$ are subintegral.
\end{enumerate}
\end{lemma}

\begin{proof}
Let us prove that 1) implies 2). We start by showing the bijection between spectra. First, see that the induced maps between $\spec(C) \to \spec(A)$ and $\spec(B) \to \spec(C)$ are surjective because the extensions are integral ( Lemma \ref{LemGoingUp}, 2) ). Now let $\p\in \spec(C)$, suppose that there exists $\q_1,\q_2 \in \spec(B)$ such that $\q_1\cap C = \q_2 \cap C = \p$.
In that case $\q_1 \cap A = \q_1 \cap C \cap A = \p \cap A$ and the same is true for $\q_2$. Since $A\inj B$ is subintegral, we have $$\q_1 \cap A = \q_2 \cap A \implies \q_1 = \q_2$$ This shows that  $\spec(B) \to \spec(C)$ is bijective. We now suppose that there exists $\p\in \spec(A)$ and $\p_1,\p_2 \in \spec(C)$ such that $\p_1 \cap A = \p_2 \cap A = \p$. By what we've just shown, we can consider some unique $\q_1,\q_2 \in \spec(B)$ such that $\q_1 \cap C = \p_1$ and $\q_2 \cap C = \p_2$. Then $\q_1 \cap A = \q_2 \cap A = \p$, so $\q_1 = \q_2$ and finally $\p_1 = \p_2$. It shows that $\spec(C) \to \spec(A)$ is bijective.\\
We now show the isomorphisms on the residue fields. Let's consider the following commutative diagram
$$\xymatrix{
   A \hspace{0.1cm} \ar@{^{(}->}[r]^{i_1} \ar@{ ->>}[d]^{\pi} & C \ar@{^{(}->}[r]^{i_2} \ar@{ ->>}[d]^{\pi'} & B \ar@{ ->>}[d]^{\pi''}\\
    A\diagup \p \hspace{0.1cm} \ar@{^{(}->}[r]^{f_1} \ar@{^{(}->}[d] & C\diagup \p' \ar@{^{(}->}[r]^{f_2} \ar@{^{(}->}[d] & B \diagup \p'' \ar@{^{(}->}[d]\\
    \kappa(\p) \hspace{0.1cm} \ar@{^{(}->}[r] & \kappa(\p') \ar@{^{(}->}[r] & \kappa(\p'') 
}$$
The diagram is indeed commutative because the $f_i$ are obtained by the universal property of quotient and their injectivity comes, for example with $f_1$, from the following equality : $$\ker(\pi' \circ i_1) = \{ a\in A \text{ such that } i(a) \in \p'\} = \p' \cap A = \p$$
By hypothesis, $A \inj B$ is subintegral, so $\kappa(\p) \to \kappa(\p'')$ is an isomorphism. This implies that $\kappa(\p) \to \kappa(\p')$ and  $\kappa(\p') \to \kappa(\p'')$ are also isomorphisms. The first implication is proved.

We now show that 2) implies 1). Let's suppose that $A\inj C$ and $C \inj B$ are subintegral. Let $\p\in \spec(A)$, then there exists a unique element $\p'\in\spec(C)$ above $\p$. In the same way, there exists a unique element $\p'' \in \spec(B)$ above $\p'$, so $\p''$ is the unique element of $\spec(B)$ above $\p$. The fact that we have $\kappa(\p') \simeq \kappa(\p)$ and $\kappa(\p'') \simeq \kappa(\p')$ gives us the equiresiduality. We have shown that $A\inj B$ is subintegral.
\end{proof}

The three following classical results brings a better understanding of the Jacobson radical in the definition of the seminormalization.
\begin{lemma}\label{LemIntersecLoc}
Let $A\inj B$ be an extension of rings, $\p\in \spec(A)$ and $\q\in \spec(B)$. Then $$ \q\cap A = \p \iff \q B_{\p} \cap A_{\p} = \p A_{\p} $$
\end{lemma}

\begin{lemma}\label{LemMaxLoc}
Let $A \inj B$ be an integral extension and $\p \in \spec(A)$. Then $$\spm(B_{\p}) = \{\q B_{\p} \mid \q\in\spec(B) \text{ and }\q\cap A = \p \}$$
\end{lemma}

\begin{rmq}
If $A\inj B$ is an integral extension and $\p\in \spec(A)$, then  $$\Rrm(B_{\p}) = \displaystyle{\bigcap_{\m\in\spm(B_{\p})}\m = \bigcap_{ \q\in \spec(B) , \q\cap A = \p } \q B_{\p} } $$
\end{rmq}

\begin{lemma}\label{LemJacobsonIntersec}
Let $A\inj B$ be an integral extension of rings. Then
$$\Rrm(A) = \Rrm(B)\cap A $$
\end{lemma}

Now, we have all the tools to demonstrate that the seminormalization of a ring in another one is its biggest subintegral subextension. But before that, we must check that the seminormalization gives a subintegral extension.
\begin{prop}\label{PropSnSub}
Let $A \inj B$ be an integral extension of rings. Then
$$A \inj A^+_B\text{ is subintegral}$$
\end{prop}

\begin{proof}
We recall that $A^+_B = \{ b\in B\text{ | }\forall \p\in \spec(A) \text{ } b_{\p}\in A_{\p} + \Rrm(B_{\p}) \} \subset B$. To clarify the proof, we will use the following rating abuse : $A^+ = A^+_B$ and $A^+_{\p} = (A^+)_{\p}$.\\
Let's start by showing the bijection between spectra. In fact, we just have to show that $\spec(A^+) \to \spec(A)$ is injective since the extension is integral. Let $\p \in \spec(A)$ and $\q_1,\q_2 \in \spec(A^+)$ such that $\q_1 \cap A = \q_2 \cap A = \p$. Suppose that there exists $b\in \q_1\cap \q_2^c$ and write $b_{\p} = \alpha + \beta \in A_{\p} + \Rrm(B_{\p})$. Then $\alpha = b_{\p} - \beta \in \q_1A^+_{\p} \cap A_{\p}$. But, by Lemma \ref{LemIntersecLoc}, we have $\q_1A^+_{\p}\cap A_{\p} = \p A_{\p} = \q_2A^+_{\p}\cap A_{\p}$ and thus $\alpha \in \q_2A^+_{\p}$. Moreover $\beta = b_{\p} - \alpha \in A^+_{\p}$ and $\beta \in \Rrm(B_{\p})$. So, by Lemma \ref{LemJacobsonIntersec}, we get $\beta \in \Rrm(B_{\p}) \cap A^+_{\p} = \Rrm(A^+_{\p}) \subset \q_2A^+_{\p}$. It gives $b_{\p} \in \q_2A^+_{\p}$ which is not possible by hypothesis on $b$. Since $\q_1$ and $\q_2$ can be inverted in the proof, we obtain $\q_1 = \q_2$ and so $\spec(A^+) \to \spec(A)$ is bijective.

We show now the equiresiduality. Let $\p \in \spec(A)$ and $\q\in \spec(A^+)$ such that $\q\cap A = \p$. Let $\q'\in \spec(B)$ such that $\q'\cap A=\p$. We can see that $\q'\cap A^+ = \q$ because $\q'\cap A^+$ is a prime ideal above $\p$ and $\spec(A^+) \to \spec(A)$ is bijective. We obtain the following commutative diagram :
$$\xymatrix{
	A_{} \ar@{^{(}->}[r] \ar@{^{(}->}[d]  & A^+ \ar@{^{(}->}[r] \ar@{^{(}->}[d] & B \ar@{^{(}->}[d]\\
	A_{\p} \ar@{^{(}->}[r] \ar@{->>}[d]^{\pi}  &  A^+_{\p} \ar@{^{(}->}[r] \ar@{->>}[d]^{\pi^+} & B_{\p} \ar@{->>}[d]^{\pi'} \\
	A_{\p} / \p A_{\p} \ar@{^{(}->}[r]^f \ar@{=}[d] & A^+_{\p} / \q A^+_{\p} \ar@{^{(}->}[r] \ar@{^{(}->}[d] & B_{\p} / \q'B_{\p} \ar@{^{(}->}[d] \\
	A_{\p} / \p A_{\p} \ar@{^{(}->}[r] & A^+_{\q} / \q A^+_{\q} \ar@{^{(}->}[r] & B_{\q'} / \q B_{\q'} }$$

We want to show that $f$ is surjective. Let $\pi^+(\frac{b}{s}) \in A^+_{\p}/\q A^+_{\p}$. Then $b\in A^+$ so there exists $\alpha \in A_{\p}$ and $\beta \in \Rrm(B_{\p})$ such that $b_{\p} = \alpha+\beta$. The remark under Proposition \ref{LemMaxLoc} tells us that $\Rrm(B_{\p}) \subset \q'B_{\p}$ so $\pi'(b_{\p}) = \pi'(\alpha)$ and since the diagram is commutative, we get $\pi^+(b_{\p}) = \pi^+(\alpha) = f\circ \pi(\alpha)$. Finally we obtain
$$f\left(\frac{\pi(\alpha)}{\pi(s)}\right) = \frac{\pi^+(b_{\p})}{\pi^+(s)} = \pi^+\left(\frac{b}{s}\right)$$
which proves the surjectivity of $f$ and gives us $A_{\p}/\p A_{\p} \simeq A^+_{\p} / \q A^+_{\p}$.\\

Finally one can see that $A^+_{\p}/\q A^+_{\p} \simeq A^+_{\q}/\q A^+_{\q}$ because
$$A^+_{\q}/\q A^+_{\q} \simeq \fracloc(A^+/\q)\text{ and }A^+/\q \inj A^+_{\p}/\q A^+_{\p} \implies A^+_{\q}/\q A^+_{\q} \inj A^+_{\p}/\q A^+_{\p}$$

This gives us the result
$$\kappa(\p) \simeq A_{\p}/\p A_{\p} \simeq A^+_{\q}/\q A^+_{\q} \simeq \kappa(\q)$$
\end{proof}
Finally, we get the main result of this section that we have already stated in Proposition \ref{PropPU}.

\begin{rappelprop}[Proposition \ref{PropPU}]
Let $A\inj C\inj B$ be integral extensions of rings. Then the following statements are equivalent :
\begin{enumerate}
    \item[1)] The extension $A \inj C$ is subintegral.
    \item[2)] The image of $C\inj B$ is a subring of $A^+_B$.
\end{enumerate}
\end{rappelprop}

\begin{proof}
Let's prove that 1) implies 2). Suppose that $A\inj C$ is subintegral, we want to show :
$$\forall c \in C\text{, }\forall \p \in \spec(A) \text{, } c_{\p}\in A_{\p} + \Rrm(B_{\p})$$
Let $\p\in \spec(A)$. We write $\q\in \spec(C)\text{ such that } \q\cap A = \p$ and so we get the following diagram :
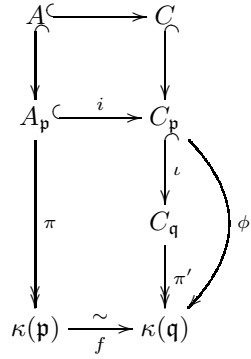
\vspace{-0.1cm}
$$\xymatrix{
   A \ar@{^{(}->}[d] \ar@{^{(}->}[r]& C \ar@{^{(}->}[d] \\
    A_{\p} \ar@{->>}[dd]^{\pi} \ar@{^{(}->}[r]^{i} &C_{\p} \ar@{^{(}->}[d]^{\iota} \ar@/^2pc/@{->}[dd]^{\phi} \\
    &C_{\q} \ar@{->>}[d]^{\pi'} \\
    \kappa(\p) \ar@{->}[r]^{\sim}_{f}&\kappa(\q)
}$$
One can see that the diagram is commutative thanks to the following equality
$$\text{Ker}(\phi)=\{c_{\p} \in C_{\p} \text{ such that }\pi'\circ\iota(c_{\p}) = 0\} = \{ c_{\p}\in C_{\p} \text{ such that } \iota(c_{\p}) \in \q C_{\q}  \} = \q C_{\q}\cap C_{\p} = \q C_{\p}.$$
Thus, by Lemma \ref{LemIntersecLoc}, we have $$\ker(\phi\circ i)=\q C_{\p}\cap A_{\p} = \p A_{\p}$$

We now consider $c \in C$ and we want to show $c \in A^+_B$. We have $\phi(c_{\p}) \in \kappa(\q)$, so one can consider $\frac{a}{s}\in A_{\p}$ such that $f\circ\pi\left(\frac{a}{s}\right) = \phi(c_{\p})$. Thanks to the diagram's commutativity, we have $f\circ \pi = \phi \circ i$, which gives us $$\phi(c_{\p}) = f\circ \pi\left(\frac{a}{s}\right) = \phi\left(i\left(\frac{a}{s}\right)\right)$$
Thus $\phi(c_{\p} - i(\frac{a}{s}))=0$ and so $c_{\p} - i(\frac{a}{s}) \in  \q C_{\p}$. But, since $\q$ is the only ideal above $\p$, we have $\q C_{\p} = \Rrm(C_{\p})$. So, by Lemma \ref{LemJacobsonIntersec}, we get $C_{\p} - i(\frac{a}{s}) \in \Rrm(C_{\p}) = \Rrm(B_{\p})\cap C_{\p} \subset \Rrm(B_{\p})$. And so $c_{\p}\in A_{\p} + \Rrm(B_{\p})$, which shows $C \subseteq A_B^+$.

We now prove that 2) implies 1). Suppose we have $A\inj C \inj A_B^+ \inj B$. Those extensions are integral and, by Proposition \ref{PropSnSub}, the extension $A\inj A_B^+$ is subintegral. Then Lemma \ref{LemTransSub} tells us that $A\inj C$ is subintegral.
\end{proof}
As said previously, Proposition \ref{PropPU} can be put into the form of a universal property. We rewrite it in the following way :

\begin{theorem}[Universal property of seminormalization]
Let $A\inj B$ be an integral extension of rings. For every intermediate extension $C$ of $A\inj B$ such that $A\inj C$ is subintegral, the image of $C$ by the injection $C\inj B$ is contained in $A_B^+$.
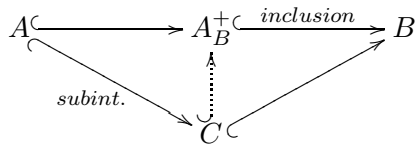
$$\xymatrix{
   A \ar@{_{(}->}[rrd]_{subint.} \ar@{^{(}->}[rr]& & A_B^+ \ar@{^{(}->}[rr]^{inclusion} & & B \\
    && C \ar@{^{(}.>}[u] \ar@{_{(}->}[urr] &&
}$$
\end{theorem}

\begin{rmq}
Let $A\inj B$ be an integral extension. We have $A \inj A^+_B$ subintegral by Proposition \ref{PropSnSub}. So we can apply the universal property in the following way :
$$
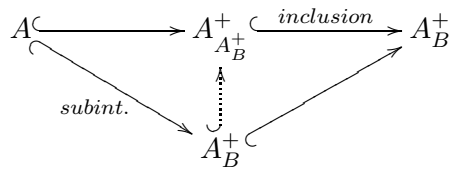
\xymatrix{
   A \ar@{_{(}->}[rrd]_{subint.} \ar@{^{(}->}[rr]& & A^+_{A^+_B} \ar@{^{(}->}[rr]^{inclusion} & & A^+_B \\
    && A^+_B \ar@{^{(}.>}[u] \ar@{_{(}->}[urr] &&
}$$
Thus $A^+_B \inj A^+_{A^+_B}$. But, by definition, $A^+_{A^+_B}$ is included in $A^+_B$. We get the following idempotency property $$ A^+_B = A^+_{A^+_B} $$
\end{rmq}

\newpage
\section{Universal property of the seminormalization in the geometric case.}\label{SectionUniversalPropGeo}

Let $k$ be an algebraically closed field of characteristic zero and $X = \spec(A)$ be an affine algebraic variety with $A$ a $k$-algebra of finite type. Let $k[X] := A$ denote the coordinate ring of $X$. We have $k[X]\simeq k[x_1,...,x_n]/I$ for an ideal $I$ of $k[x_1,...,x_n]$ and we will always assume $I$ to be radical. We recall that $X$ is irreducible if and only if $k[X]$ is a domain. A morphism $\pi : Y \to X$ between two varieties induces the morphism $\pi^* : k[X] \to k[Y]$ which is injective if and only if $\pi$ is dominant. We say that $\pi$ is of finite type (resp. is finite) if $\pi^*$ makes $k[Y]$ a $k[X]$-algebra of finite type (resp. a finite $k[X]$-module). 

The space $X$ is equipped with the Zariski topology for which the closed sets are of the form $\Vcal(I) := \{ \p \in \spec(k[X])\mid I\subset \p \}$ where $I$ is an ideal of $k[X]$. We define $X(k) := \{ \m \in \spm(k[X])\mid \kappa(\m) = k \}$. Thus, if we write $k[X] = k[x_1,...,x_n] / I$, the elements of $X(k)$ can be seen as elements of $\spm(k[x_1,...,x_n])$ containing $I$. The Nullstellensatz gives us a Zariski homeomorphism between $X(k)$ and the algebraic set $\Zcal(I) := \{ x\in k^n \mid \forall f \in I\text{ }f(x)=0 \} \subset k^n$. We will call $\pik : Y(k) \to X(k)$ the restriction of $\pi$ to $Y(k)$. We will add the prefix «$Z-$» before a property if it holds for the Zariski topology. If $X$ is irreducible, then we write $\K(X) := \fracloc(k[X])$.

If needed, for $x\in X$ (resp. $X(k)$), we will write $\p_x$ (resp. $\m_x$) its associated ideal in $k[X]$.\\

The goal of this section is to prove the following theorem and to write a universal property of the seminormalization for affine varieties.

\begin{theorem}\label{TheoSubSsiBij}
Let $\pi : Y \to X$ be a finite morphism between affine varieties. Then the following properties are equivalent.
\begin{enumerate}
    \item[1)] The morphism $\pik : Y(k) \to X(k)$ is bijective.
    \item[2)] The extension $\pi^* : k[X] \inj k[Y]$ is subintegral.
    \item[3)] The morphism $\pi : Y \to X$ is a Z-homeomorphism.
    \item[4)] The morphism $\pik : Y(k) \to X(k)$ is a Z-homeomorphism.
\end{enumerate}
\end{theorem}

The following classical result will be useful in the sequel.

\begin{lemma}\label{LemRamifié}
Let $\pi : Y \to X$ be a finite morphism between irreducible varieties. Then there exists a non empty Z-open subset $U$ of $X$ such that $$\forall u\in U(k)\text{, }\#\pik^{-1}(u) = [\K(Y):\K(X)]$$
\end{lemma}

\begin{proof}
Let $U$ be a non empty Z-open set of $X$ such that $U$ is normal. Since the characteristic of $k$ is zero, the extension $\K(X) \inj \K(U)$ is separable. So we can apply Theorem 7 p.117 from \cite{Shafa} on the morphism $\pi$ restricted to $\pi^{-1}(U)$.
\end{proof}

The following proposition gives the main part of Theorem \ref{TheoSubSsiBij}. The Nullstellensatz is the key to prove it, which is not surprising for a statement of this kind.
\newpage
\begin{prop}\label{PropSubEqBij}
Let $\pi : Y \to X$ be a finite morphism between affine varieties. The following properties are equivalent
\begin{enumerate}
    \item[1)] The extension $\pi^* : k[X] \inj k[Y]$ is subintegral.
    \item[2)] The morphism $\pik : Y(k) \to X(k)$ is bijective.
\end{enumerate}
\end{prop}

\begin{proof}
Suppose that $k[X] \inj k[Y]$ is subintegral, then $Y \to X$ is a bijection and as the extension is integral, the inverse image of $X(k)$ is $Y(k)$ (cf. lemma \ref{LemGoingUp}). Therefore 
$$\pik : Y(k) \to X(k)\text{ is bijective}$$

Conversely, suppose that $Y(k) \to X(k)$ is bijective. We start by checking that $\pi$ is bijective. Let $y_1, y_2 \in Y$ such that $\pi(y_1) = \pi(y_2)$ and write $\p_{y_1}$, $\p_{y_2}$ for the associated prime ideals of $k[Y]$. We want to prove that $y_1 = y_2$. By the Nullstellensatz, it is equivalent to show 
$$\p_{y_1} \text{\hspace{0.3cm} = } \displaystyle \bigcap_{\p_{y_1} \subseteq \m_y\in Y(k)} \m_y \text{\hspace{0.3cm} = } \displaystyle \bigcap_{\p_{y_2} \subseteq \m_y\in Y(k)} \m_y \text{\hspace{0.3cm} = \hspace{0.3cm} } \p_{y_2}$$
Let $y\in Y(k)$ such that $\p_{y_1} \subset \m_y$. We then have $\p_{\pi(y_1)} \subset \m_{\pi(y)}$, so $\p_{\pi(y_2)} \subset \m_{\pi(y)}$. Thus, by the going-up property, we can consider $y'\in Y(k)$ such that $\p_{y_2} \subseteq \m_{y'}$ and $\pi(y) = \pi(y')$. As we have supposed that $\pi$ is injective on $Y(k)$, we get $y = y'$.\\
Finally $$\forall y \in Y(k) \text{ \hspace{0.8cm}} \p_{y_1}\subseteq \m_y \implies \p_{y_2} \subseteq \m_y$$
The role of $y_1$ and $y_2$ can be inverted in the proof, so we get
$$\p_{y_1} = \displaystyle \bigcap_{y\in Y(k), \p_{y_1} \subseteq \m_y} \m_y \text{  \hspace{0.3cm}= } \displaystyle \bigcap_{y\in Y(k), \p_{y_2} \subseteq \m_y} \m_y = \p_{y_2} \text{  \hspace{0.4cm}i.e. \hspace{0.2cm} } y_1 = y_2$$

It remains to check the equiresiduality. Let $x\in X$ and $y\in Y$ such that $\pi(y) = x$. We want to prove $\kappa(\p_x) \simeq \kappa(\p_y)$. If we write $V = \spec(k[X]/\p_x)$ and $W = \spec(k[Y]/\p_y)$, we get the following commutative diagram :

$$\xymatrix{
    k[X] \ar@{->>}[d]^{\pi_X} \ar@{^{(}->}[r]^{\pi^*}& k[Y] \ar@{->>}[d]^{\pi_Y} \\
    k[V] \ar@{^{(}->}[d] \ar@{^{(}->}[r]^{(\pi_{|_W})^*} &k[W] \ar@{^{(}->}[d] \\
     \K(V) \ar@{^{(}->}[r] & \K(W)
}$$

As $k[Y]$ is a finite $k[X]$-module, we have that $k[W]$ is a finite $k[V]$-module. Thus $\pi_{|W}$ is a finite morphism between two irreducible varieties. Therefore we can apply Lemma \ref{LemRamifié} to consider a non empty Z-open set $U$ of $W$ such that
$$\forall u\in U(k) \text{\hspace{0.5cm}} \#(\pi_{|W(k)})^{-1}(u)=[\K(W):\K(V)]$$
But we have previously shown that $\pi$ is bijective so $[\K(W):\K(V)] = 1$ and so we get $\kappa(\p_x) \simeq \kappa(\p_y)$.
\end{proof}

\newpage
The following lemma completes the results needed to get Theorem \ref{TheoSubSsiBij}.
\begin{lemma}\label{LemFiniBijDoncHomeoSurSpec}
Let $A \inj B$ be an integral extension of rings and $\pi : \spec(B) \to \spec(A)$ be the induced map. If the morphism $\pi$ is bijective, then $\pi$ is a Z-homeomorphism.
\end{lemma}

\begin{proof}
Since $\pi$ is Z-continuous, even when the extension is not integral, we just have to show that $\pi$ is Z-closed. Let $\p\in \spec(A)$ and $\q\in\spec(B)$ such that $\q\cap A = \p$. We have
$$\pi(\mathcal{V}(\q)) = \{ \q'\cap A\mid \q\subset \q' \in \spec(B) \}$$
so $\pi(\mathcal{V}(\q)) \subset \mathcal{V}(\q\cap A) = \mathcal{V}(\p)$. If $\p'\in \mathcal{V}(\p)$, then the going-up property says that there exists $\q'\in \spec(B)$ such that $\q' \cap A = \p'$ and $\q\subset \q'$. Therefore $\mathcal{V}(\p) \subset \pi(\mathcal{V}(\q))$ and so $\pi(\mathcal{V}(\q)) = \mathcal{V}(\p)$ which is Z-closed.
\end{proof}

We can finally prove the main theorem of this section.
\begin{rappel}[Theorem \ref{TheoSubSsiBij}]
Let $\pi : Y \to X$ be a finite morphism between two affine varieties. Then the following properties are equivalent.
\begin{enumerate}
    \item[1)] The morphism $\pik : Y(k) \to X(k)$ is bijective.
    \item[2)] The extension $\pi^* : k[X] \inj k[Y]$ is subintegral.
    \item[3)] The morphism $\pi : Y \to X$ is a Z-homeomorphism.
    \item[4)] The morphism $\pik : Y(k) \to X(k)$ is a Z-homeomorphism.
\end{enumerate}
\end{rappel}

\begin{proof}
First of all, Proposition \ref{PropSubEqBij} gives us the first equivalence. Now suppose that $k[X] \inj k[Y]$ is subintegral. Then $\pi$ is bijective and since $k[X] \inj k[Y]$ is integral, we can apply Lemma \ref{LemFiniBijDoncHomeoSurSpec} and get that $\pi$ is a $Z$-homeomorphism. Finally, if $\pi$ is a Z-homeomorphism then $\pik$ is a Z-homeomorphism and thus $\pik$ is bijective.
\end{proof}

Before going further, we must define the seminormalization of an affine variety. If $\pi : Y \to X$ is a finite morphism between two affine varieties, then $k[X]^+_{k[Y]}$ is a finite $k[X]$-module because it is a submodule of $k[Y]$. Thus $k[X]^+_{k[Y]}$ is a $k$-algebra of finite type because so is $k[X]$. 
\begin{definition}\label{DefSeminormVariete}
Let $\pi : Y \to X$ be a finite morphism between two affine varieties. The affine variety defined by
$$X^+_Y = \spec(k[X]^+_{k[Y]})$$
is called the seminormalization of $X$ in $Y$.
\end{definition}
As a consequence of the previous theorem, we are going to see that, in order to get the seminormalization of a variety in an other one, it is sufficient to glue together the closed points in the fibers of $\pik$.
\begin{definition}
Let $A \inj B$ be an integral extension of rings.\\
We define
$$A^{+_{max}}_B = \{b \in B\mid\forall \m \in \spm(A) \text{, } b_{\m} \in A_{\m}+\Rrm(B_{\m})\}$$
\end{definition}

\begin{cor}\label{CorSemiMaxEstSemi}
Let $\pi : Y \to X$ be a finite morphism between two affine varieties.\\ Then
$$ k[X]^{^{+_{max}}}_{k[Y]} = k[X]^{^+}_{k[Y]} $$
\end{cor}

\begin{proof}
We write $A := k[X]$ and $B := k[Y]$. So we have to show $A_B^{+_{max}} = A^+_B$. First of all, note that the inclusion $A^+_B \subseteq A_B^{+_{max}}$ is obvious.

In order to prove the other inclusion, we have to show that $\spm(A_B^{+_{max}}) \to \spm(A)$ is bijective. If we succeed, we will be able to deduce, thanks to Proposition \ref{PropSubEqBij}, that $A\inj A_B^{+_{max}}$ is subintegral and the universal property of the seminormalization will give us the result.

Let $\m \in \spm(A)$ and $\m_1,\m_2 \in \spm(A_B^{+_{max}})$ such that $\m_1 \cap A = \m_2 \cap A = \m$. Suppose that there exists $b\in \m_1\cap \m_2^c$ and write $b_\m = \alpha + \beta \in A_\m + \Rrm(B_\m)$. Then $\beta = b_{\m} - \alpha \in A^{+_{max}}_{B,\m}$, so $\beta \in \Rrm(B_{\m})\cap A^{+_{max}}_{B,\m}$. By Lemma \ref{LemJacobsonIntersec}, we get $\beta \in \Rrm(A^{+_{max}}_{B,\m}) \subset \m_1A^{+_{max}}_{B,\m}\cap \m_2A^{+_{max}}_{B,\m}$. In particular, $\beta \in \m_1A^{+_{max}}_{B,\m}$ and by hypothesis on $b$, we also have $b_{\m} \in \m_1A^{+_{max}}_{B,\m}$. 
We then have $\alpha = b_\m - \beta \in \m_1A^{+_{max}}_{B,\m} \cap A_\m$ and $\m_1 A^{+_{max}}_{B,\m}\cap A_{\m} = \m A_\m = \m_2A^{+_{max}}_{B,\m}\cap A_{\m}$, by Lemma \ref{LemIntersecLoc}. Thus $\alpha \in \m_2A^{+_{max}}_{B,\m}$ and since $\beta \in \m_2A^{+_{max}}_{B,\m}$, we get $b_\m \in \m_2A^{+_{max}}_{B,\m}$ which is not possible by assumption on $b$. Considering the fact that the role of $\m_1$ and $\m_2$ can be inverted in the proof, we get $\m_1 = \m_2$.
\end{proof}

We are now able to write a geometric version of the universal property of seminormalization. By just adapting the one we have obtained at the end of the first section with coordinate rings, we get :

$$\xymatrix{
    k[X] \ar@{_{(}->}[rrd]_{subint.} \ar@{^{(}->}[rr]& & k[X]_{k[Y]}^+\simeq k[X^+_Y] \ar@{^{(}->}[rr] & & k[Y] \\
    && k[Z] \ar@{.>}[u] \ar@{_{(}->}[urr] &&
}$$

This gives us the following statement for varieties :
\begin{theorem}[Universal property of the seminormalization]\label{propPUgéo}
Let $Y \to Z \to X$ be finite morphisms of affine varieties. Then $\pik : Z(k) \to X(k)$ is bijective if and only if there exists a morphism $\pi_Z^+ : X_Y^+ \to Z$ such that $\pi \circ \pi^+_Z = \pi^+$. Moreover $\pi_Z^+$ is unique and $\pi_{Z(k)}^+ : X^+_Y(k) \to Z(k)$ is bijective.
$$\xymatrix{
    Y \ar[rrd] \ar[rr]&& X^+_Y \ar@{.>}[d]^{\pi^+_Z} \ar[rr]^{\pi^+} && X \\
    && Z \ar[urr]_{\pi bij} &&
}$$
\end{theorem}

\begin{proof}
We have the following equivalences :
$$\begin{array}{rcl}
    \piC : Z(k) \to X(k) \text{ bijective} & \iff & k[X] \inj k[Z] \text{ subintegral} \text{ ( by Theorem \ref{TheoSubSsiBij} )}\\
     & \iff & k[Z] \subset k[X]_{k[Y]}^+ \simeq k[X_{k[Y]}^+] \text{ ( by Proposition \ref{PropPU} )}\\
     & \iff & \exists \pi_Z^+ : X_Y^+ \to Z \text{ dominant, such that } \pi \circ \pi^+_Z = \pi^+
\end{array}$$
We get the uniqueness of $\pi_Z^+$ by injectivity of $\pi$. Since $k[X] \inj k[X^+_Y]$ is subintegral then Lemma \ref{LemTransSub} says that $k[Z] \inj k[X^+_Y]$ is also subintegral. So Theorem \ref{TheoSubSsiBij} tells us that $\pi_{Z(k)}^+ : X^+_Y(k) \to Z(k)$ is bijective.
\end{proof}

\section{Continuous rational functions on complex affine varieties.}\label{SectionContinuousRational}

This section is dedicated to the introduction and study of continuous rational functions on affine varieties and to show that they are linked with the concept of seminormalization. Since we use the Euclidean topology, we restrict ourself to complex affine varieties. A generalization on any algebraic variety over algebraic closed field of characteristic zero is given in the forthcoming paper \cite{BFMQ}.
The first subsection is dedicated to the study of the ring of continuous rational functions. Concretely, we show that this ring corresponds to the coordinate ring of the seminormalization of the variety.
In the second subsection, we show that continuous rational functions are always regulous in the case of complex affine varieties. The theory of regulous functions comes from real algebraic geometry and was introduced in \cite{FHMM} and \cite{Kollar}. In this theory, continuous rational functions and regulous functions are not the same on real singular algebraic sets.
In the third subsection we look at continuous functions which are a root of a polynomial in $\C[X][t]$. This leads to identify continuous rational functions with "c-holomorphic" functions with algebraic graphs on algebraic varieties. This kind of functions is studied in \cite{BDTW} and \cite{BDT} in the point of view of complex analytic geometry. We show, with algebraic arguments, that they coincide on
algebraic varieties. Finally, the fourth subsection presents examples of continuous rational functions. 

We begin by recalling classical results on the normalization of an affine variety.

\begin{definition}\label{DefAnneauTotalDesFrac}
Let $A$ be a commutative ring. We note $\K$ the localization $S^{-1}A$ where $S$ is the set of non-zero-divisors $A$. The ring $\K$ is called the \textit{total ring of fractions} of $A$.\\

If $A$ is reduced with a finite number of minimal prime ideals $\p_1,...,\p_n$, then $$\p_1 \cap ... \cap \p_n = \bigcap_{p\in \spec(A)} p = (0)$$
We get the following injections :
$$ A\inj A / \p_1 \times ... \times A / \p_n \inj K_1 \times ... \times K_n \text{ where }K_i := \fracloc(A/\p_i) $$
Then the total ring of fractions of $A$ corresponds to the product of the fields $K_i$.
\end{definition}
We define $A'$ to be the integral closure of $A$ in $\K$ and we simply call it \textit{the integral closure of $A$}. In the same spirit, the seminormalization of $A$ in $\K$ is denoted by $A^+$ and is simply called \textit{the seminormalization of $A$}. Finally, we say that $A$ is \textit{seminormal} if $A^+=A$.

For $X$ an affine variety, the total ring of fractions of $\C[X]$ is denoted by $\K(X)$. The ring $\K(X)$ (which is a field when $X$ is irreducible) is also the ring of classes of rational fractions on $X$ and is called the ring of rational functions on $X$. It means that it represents the set of classes of regular functions $f$ on a Z-dense Z-open set $U$ of $X(\C)$ with the equivalence relation $(f_1, U_1) \sim (f_2, U_2)$ iff $f_1 = f_2$ on $U_1 \cap U_2$.

We say that a morphism $\varphi : Y \to X$ is birational if the associated morphism $\K(X) \inj \K(Y)$ is an isomorphism.

The integral closure of $\C[X]$ being a finite $\C[X]$-module ( see \cite{Eisen} Thm 4.14 ), it is also a $\C$-algebra of finite type. Thus we can define the \textit{normalization} $X'$ of $X$ such that $X' = \spec(\C[X]')$. We get a finite and birational morphism $\pi' : X' \to X$. The normalization of $X$ is the biggest affine variety finitely birational to $X$. It means that for every finite, birational morphism $\varphi : Y\to X$, there exists $\psi : X' \to Y$ such that $\pi' = \varphi \circ \psi$.

The seminormalization $X^+$ of $X$ is the variety $X^+_{X'}$ defined in Definition \ref{DefSeminormVariete}. The seminormalization $X^+$ comes with a finite, birational and bijective morphism ${\pi^+ : X^+ \to X}$ whose universal property is given by Proposition \ref{propPUgéo}.

Finally for every affine variety $X$, we have the following extensions of rings :
$$ \C[X] \inj \C[X^+] \inj \C[X'] \inj \K(X). $$ 


\begin{definition}
Let $X$ be an affine variety. We write $\KO(X(\C))$ the set of continuous functions $f : X(\C) \to \C$ for the Euclidean topology which are rational on $X(\C)$.
\end{definition}

\begin{ex}
The most classical example of such a function is the following :\\
Let $X = \spec(\C[x,y]/<y^2-x^3>)$, consider the function $f$ defined on $X(\C) = \Zcal(y^2-x^3)$ by 
$$f = \left\{ 
\begin{array}{l}
    \frac{y}{x} \text{ if }x\neq 0 \\
    0 \text{ else} \end{array} 
    \right.$$
\end{ex}

First of all, we show that we will always be able to assume $X$ to be irreducible thanks to the following lemmas.

\begin{lemma}\label{LemUnionDenseEstDense}
Let $E$ be a topological space and $\{E_i\}_{i\in I}$ be a covering of $E$. Let $A$ be a subspace of $E$ such that $A\cap E_i$ is dense in $E_i$ for all $i\in I$. Then $A$ is dense in $E$.
\end{lemma}

\begin{proof}
Let $U$ be a non-empty open set of $E$. Then there exists $i\in I$ such that $U\cap E_i \neq \varnothing$. Since $A\cap E_i$ is dense in $E_i$ and $U\cap E_i$ is a non-empty open set of $E_i$, we get $A\cap U \cap E_i \neq \varnothing$. Then $A\cap U \neq \varnothing$. So $A$ is dense in $E$.
\end{proof}

\begin{lemma}\label{Lemirreducible}
Let $X$ be an affine variety and $f:X(\C) \to \C$. We write $X = \displaystyle\bigcup_{i=1}^n X_i$ where the $X_i$ are the irreducible components of $X$. The following properties are equivalent
\begin{enumerate}
    \item[1)] $f\in \KO(X(\C))$
    \item[2)] $\forall i\in \llbracket 1;n \rrbracket \hspace{0.3cm} f_{|X_i(\C)} \in \KO(X_i(\C))$
\end{enumerate}

\end{lemma}

\begin{proof}
Let $f \in \KO(X(\C))$, we call $U$ the Z-dense Z-open set on which $f$ is regular. For $j\in \llbracket1;n\rrbracket$, the set $ X\backslash \bigcup_{i\neq j} X_i $ is a Z-open set contained in $X_j$, which is not empty because $X_j \nsubseteq \bigcup_{i\neq j} X_i$. Thus, since $U$ is Z-dense, we have $X_j(\C)\cap U \neq \varnothing$. So $f$ is regular on $X_j(\C) \cap U$ which is Z-dense because $X_j$ is irreducible. Then, $f_{|X_j(\C)}$ being clearly continuous, we have $f_{|X_j(\C)}\in \KO(X_j(\C))$.\\

Let $f : X(\C) \to \C$ such that $f_{|X_i(\C)} \in \KO(X_i(\C))$ for all $i\in \llbracket 1;n \rrbracket$. Then $f$ is regular on a Z-dense Z-open set of every component and since, by Lemma \ref{LemUnionDenseEstDense}, a union of dense open sets of each irreducible component of $X$ is a dense open set of $X$, we get that $f$ is regular on a Z-dense Z-open set of $X(\C)$. Let's show that $f$ is continuous on $X(\C)$. First of all, $f$ is clearly continuous on every point of the set $$X(\C)\backslash \left( \bigcup_{i\neq j} X_i(\C) \cap X_j(\C) \right)$$ which are the points that do not belong to any intersection of components. Now let's take a look at the continuity near the other points. Let $x \in \bigcap_{j\in J} X_j(\C)$ where $J$ is a subset of $\llbracket1;n\rrbracket$. Let $\epsilon > 0$ and $j\in J$. Since $f$ is continuous on $X_j(\C)$, we can consider a Euclidean open set $U_j$ containing $x$ such that $\forall y\in X_j(\C)\cap U_j$ we have $|f(x)-f(y)|<\epsilon$. By doing the same for all $j\in J$, we obtain 
$$ \forall y\in X(\C)\cap \left( \bigcap_{j\in J} U_j \right) \hspace{1cm} |f(x)-f(y)|<\epsilon $$
\end{proof}

\begin{rmq}
Let $X$ be an affine variety and $X = \displaystyle\bigcup_{i=1}^n X_i$ its decomposition into irreducible components. Then
$$\KO(X(\C)) \simeq \{(f_1,...,f_n)\in \KO(X_1(\C))\times ... \times \KO(X_n(\C)) \mid f_{i_{\mid X_i(\C)\cap X_j(\C)}} = f_{j_{\mid X_i(\C)\cap X_j(\C)}} \}$$
\end{rmq}

The next proposition shows that the continuity allows us to be more precise concerning the Z-open set where an element of $\KO(X(\C))$ is regular. We write $X_{reg}$ (resp. $X_{sing}$) the set of regular (resp. singular) points of $X$ and $X_{reg}(\C)$ (resp. $X_{sing}(\C)$) those of $X(\C)$.

\begin{prop}\label{PropRatEstRegSurRegX}
A function belongs to $\KO(X(\C))$ if and only if it is continuous for the Euclidean topology and it is regular on $X_{reg}(\C)$.
\end{prop}

\begin{rmq}
Let $x\in X$, we write $\Ocal_{X,x} := \C[X]_{\p_x}$ the ring of functions which are regular at $x$. If $X$ is irreducible and $W$ is a subvariety of $X$, we write $\Ocal_X(W) := \cap_{x\in W} \Ocal_{X,x}$.
\end{rmq}

\begin{proof} We assume $X$ irreducible thanks to Lemma \ref{Lemirreducible}. Let $f : X(\C) \to \C$ be regular on the Z-dense Z-open set $U(\C)$ and continuous on $X(\C)$. Then there exists $q,p \in \C[X]$ such that $pf = q$ on $U(\C)$. As a Z-dense Z-open set is dense for the Euclidean topology and $pf - q$ is continuous, we get $pf - q = 0$ on $X(\C)$.\\
Let $x\in X_{reg}(\C)$, we have to show $f \in \Ocal_{X,x}$. If $p$ is a unit in $\Ocal_{X,x}$, then $f = q.p^{-1} \in \Ocal_{X,x}$.
Else, since $x\in X_{reg}(\C)$, the Auslander-Buchsbaum theorem tells us that $\Ocal_{X,x}$ is a UFD. So, even if it means multiplying $q$ by some unit elements of $\Ocal_{X,x}$, we can write $$p = \prod_{i=1}^np_i^{s_i}$$ with $p_i$ some prime elements of $\Ocal_{X,x}$.
We now consider a $Z$-open neighbourhood $W_1(\C)$ of $x$ such that $pf=q$ on $W_1(\C)$ and $p_1$ is a prime element of $\Ocal_X(W_1(\C))$. Thus $q$ vanish on $\Zcal(p_1)$ and since the Nullstellensatz tells us that $\Jcal(\Zcal(p_1)) = p_1\Ocal_X(W_1(\C))$, we have $q\in p_1\Ocal_X(W_1(\C))$. So there exists $q_1 \in \Ocal_X(W_1(\C))$ such that $q = p_1q_1$. Then we obtain 
$$\prod_{i=2}^n p_i^{s_i}p_1^{s_1-1}f_{|_{W_1(\C)}} = q_{1|_{W_1(\C)}} $$
In the case where $s_1>1$ we'll have $\Zcal(p_1) \subset \Zcal(q_1)$ so we can iterate the process and get $q_{s_1} \in \Ocal_X(W_1(\C))$ such that 
$$\prod_{i=2}^n p_i^{s_i}f_{|_{W_1(\C)}} = q_{s_1|_{W_1(\C)}}$$
If we take $W_2(\C)$ a $Z$-open neighbourhood of $x$ on which $p_2$ is prime, we can the repeat the previous argument on $W_1(\C)\cap W_2(\C)$. Thus, by doing this $n$ times, we can consider a $Z$-open neighbourhood $W_n(\C)$ of $x$ in $X(\C)$ and $q_{\Sigma s_1} \in \Ocal_X(W_n(\C))$ such that $f_{|_{W_n(\C)}} = q_{\Sigma s_i|_{W_n(\C)}}$. Then we finally conclude that $f \in \Ocal_{X,x}$.
\end{proof}

\begin{rmq}
Note that the Euclidean continuity is essential in the proof to apply the argument of density at the end of the first paragraph. In particular, the Zariski continuity used in \cite{LV} wouldn't be enough because it doesn't allow us in general to extend an equality which is true on a Z-dense set. As an example we can consider the function $f$ defined on $\A^1(\C)$ by
$$f = \left\{ 
\begin{array}{l}
    \frac{1}{z} \text{ if }z\neq 0 \\
    0 \text{ else} \end{array} 
    \right.$$
This function is Z-continuous because it is a bijection and the Z-open sets of $\A^1(\C)$ are of the form $\A^1(\C)\backslash\{\text{finite nb of points}\}$. However, even if $zf(z) = 1$ on the Z-dense Z-open set $\A^1(\C)\backslash\{0\}$, this equality does not extend on the whole space $\A^1(\C)$.\\
\end{rmq}

\subsection{Connection between continuous rational functions and seminormalization.}\label{ConnectionBetweenRat}

The goal of this subsection is to study the ring $\KO(X(\C))$ for $X$ an affine variety. The main result being Proposition \ref{TheoRatioContEgalPolySN} which tells us that this ring is in fact the coordinate ring of the seminormalization of $X$, in other words $\KO(X(\C)) = \C[X^+]$. To do this, we must look at how the continuous rational functions behave when they are composed with finite morphisms of affine varieties. 

As we have seen previously, the functions in $\KO(X(\C))$ are regular on the regular points of $X(\C)$. Thus, if $X$ is normal, the singular locus is too thin for a continuous rational function not to be polynomial. This allows us to identify the ring $\KO(X(\C))$ when $X$ is normal. 
\begin{prop}\label{PropRegSurNormalSontPoly}
Let $X$ be an affine normal variety. Then $\KO(X(\C)) = \Ocal_X(X(\C)) = \C[X]$.
\end{prop}

\begin{proof}
First $\C[X] \subset \KO(X(\C))$ is obvious. Conversely, let $f\in \KO(X(\C))$. By the previous proposition we get $f\in \Ocal_X(X_{reg}(\C))$. But since $X$ is normal we have codim($X_{sing}(\C)$)$\geqslant 2$. Thus, by \cite{I} p.124, there exists a function $\tilde{f} \in \Ocal_X(X(\C))$ which coincides with $f$ on $X_{reg}(\C)$. As the function $f-\tilde{f}$ is continuous for the Euclidean topology and vanish on $X_{reg}(\C)$ which is a dense open set of $X(\C)$, we get $f = \tilde{f} \in \Ocal_X(X(\C)) = \C[X]$.
\end{proof}

In the following proposition, we improve the result of Proposition \ref{PropRatEstRegSurRegX}. We write $\Norm(X) := \{ x\in X\text{ | }\Ocal_{X,x}\text{ is integrally closed} \}$ and $\Norm(X(\C)) = \Norm(X)\cap X(\C)$.

\begin{prop}\label{PropRegSurPtsNormaux}
Let $X$ be an affine variety and $f \in \KO(X(\C))$. Then $$\forall x\in \Norm(X(\C)) \hspace{0.2cm}\text{  } f\in \Ocal_{X,x}$$
\end{prop}

\begin{proof}
Let $\pi' : X' \to X$ be the normalization morphism of $X$ and $f\in \KO(X(\C))$. Proposition \ref{PropRegSurNormalSontPoly} tells us that  $f\circ\pi' \in \C[X']$. Let $x\in \Norm(X(\C))$, then there exists a unique $x'\in X'$ such that $\pi'(x')=x$ and so $\Ocal_{X,x} \inj \Ocal_{X',x'}$. Since normalization commutes with localisation (one can see \cite{AM} for example), we get $\Ocal_{X,x}' \simeq \Ocal_{X',x'}$. But $x \in \Norm(X(\C))$ implies $\Ocal_{X,x} = \Ocal_{X,x}'$ so $$ \Ocal_{X,x} \simeq \Ocal_{X',x'} $$
Finally, since $f\circ \pi' \in \Ocal_{X',x'}$, we get $f\in \Ocal_{X,x}$.
\end{proof}

\begin{rmq}
A UFD being integrally closed, we have $X_{reg} \subset \Norm(X)$. So Proposition \ref{PropRegSurPtsNormaux} implies Proposition \ref{PropRatEstRegSurRegX}.
\end{rmq}

Before continuing the study of continuous rational functions, we have to establish some properties of finite morphisms that we will need later.
\begin{lemma}\label{LemMorphFiniDoncFerme}
Let $\pi : Y \rightarrow X$ be a finite morphism of affine varieties. Then 
$$\pik : Y(\C) \rightarrow X(\C)\text{ is surjective and closed for the Euclidean topology.}$$
\end{lemma}

\begin{proof}
Proposition \ref{LemGoingUp} tells us that the morphism $\pi$ is surjective and that the inverse image of $X(\C)$ is $Y(\C)$. It gives us the surjectivity of $\piC : Y(\C) \rightarrow X(\C)$.\\
To show that $\piC$ is closed, we are going to prove that it is proper for the Euclidean topology.
Let $K$ be a compact subset of $X(\C)$. We first have that $\piC^{-1}(K)$ is closed because $\piC$ is continuous. Suppose $\C[Y]=\C[y_1,...,y_n]/I_Y$ and $Y(\C) \subset \C^n$.
We write $Y_i : y \mapsto y_i$ the map giving the $i^{th}$ coordinate of an element of $Y(\C)$. We then have $Y_i\in \C[Y]$ and, since by hypothesis $\C[Y]$ is a finite $\C[X]$-module, there exists an identity of the form
$$ Y_i^k + a_{k-1}\circ \piC. Y_i^{k-1}+...+a_0\circ \piC = 0 \text{ with  }a_i \in \C[X] $$
Let $y \in \piC^{-1}(K)$. We write $\piC(y)=x$. If $y_i\neq 0$, then

$$\begin{array}{ll}
     & Y_i^k(y) + a_{k-1}\circ \piC(y). Y_i^{k-1}(y)+...+a_0\circ\piC(y) = 0  \\ [0.2cm]
     \implies & y_i^k+a_{k-1}(x)y_i^{k-1} + ... + a_0(x) = 0 \\[0.2cm]
     \implies & 1+a_{k-1}(x)/y_i + ... + a_0(x)/y_i^k = 0
\end{array} $$

As $K$ is a compact set, the $a_j(K)$ are bounded. So 
$$\forall (y_n)_n\in \piC^{-1}(K)^{\mathbb{N}}\text{\hspace{0.5cm}} Y_i(y_n) \nrightarrow +\infty$$
This means that $Y_i(\piC^{-1}(K))$ is bounded for all $i$ and so that $\piC^{-1}(K)$ is a compact set. More generally, we have shown :
\begin{center}
    $\piC : Y(\C) \to X(\C)$ is proper for the Euclidean topology. 
\end{center} 
Since a proper continuous map is closed, the lemma is proved.
\end{proof}

\begin{rmq}
What we have just shown implies that for every finite morphism $\pi : Y \to X$ of affine varieties with $\piC$ bijective, the morphism $\piC$ is an Euclidean homeomorphism.
\end{rmq}

Henceforth, for any given morphism $\pi : Y \to X$ of affine varieties, we shall write $\varphi : \KO(X(\C)) \to \Ccal(Y(\C),\C)$ the map $f \mapsto f\circ \piC$. The purpose of this notation is to differentiate this map from the morphism $\pi^* : \C[X] \to \C[Y]$. We will see that if $\pi$ is a finite morphism we can determine the image of $\varphi$. This will be useful for us since the normalization and seminormalization morphisms are finite.

\begin{lemma}\label{lemPhiInj}
Let $\pi : Y \rightarrow X$ be a surjective morphism of affine varieties. Then $\varphi$ is injective and $$\im(\varphi) \subset \{ f\in \KO(Y(\C))\text{ with $f$ constant on the fibers of }\piC \}$$
\end{lemma}

\begin{proof}
If $f\in \KO(X(\C))$ then we can consider $p,q \in \C[X]$ with $q$ a non-zero-divisor such that $q.f-p=0$ on $X(\C)$. Write $Y = \bigcup_{i=1}^n Y_i$ the decomposition in irreducible component of $Y$. Then, for all $i\in \llbracket 1;n \rrbracket$, we get $q\circ\pi_{\mid_{Y_i(\C)}}.f\circ\pi_{\mid_{Y_i(\C)}} - p\circ\pi{_{\mid_{Y_i(\C)}}} =0$ on $Y_i(\C)$. Thus $f\circ\pi_{\mid_{Y_i(\C)}} \in \K(Y_i)$ and since $\piC$ is continuous, we have $f\circ\pi{_{\mid_{Y_i(\C)}}}\in \KO(Y_i(\C))$. By Lemma \ref{Lemirreducible}, we get $f\circ\piC \in \KO(Y(\C))$. Moreover
$$\forall y_1,y_2 \in Y(\C) \text{\hspace{0.5cm} } \piC(y_1)= \piC(y_2) \implies f\circ\piC(y_1)= f\circ\piC(y_2)$$
which shows the lemma's inclusion. It remains to prove the injectivity of $\varphi$ :\\
Let $f \in \KO(X(\C))$ be such that $\varphi(f) = f\circ \piC = 0$ and take $x\in X(\C)$. As $\pi$ is surjective, there exists $y\in Y(\C)$ such that $\pi(y)=x$. So we can write $f(x) = f\circ\piC(y) = 0$. Thus we get $f=0$ which shows that $\varphi$ is injective.
\end{proof}

We deduce from Lemma \ref{lemPhiInj} that continuous rational functions on an affine variety can be seen as polynomial functions on its normalization.

\begin{prop}\label{PropRatiContEstEntiere}
Let $X$ be an affine variety and $f\in \KO(X(\C))$. Then $f\text{ is integral on }\C[X]$.
\end{prop}

\begin{proof}
Let $f\in \KO(X(\C))$ and $\pi : X' \to X$ be the normalization morphism of $X$. Since $\pi$ is a finite morphism, the previous lemma gives us $\KO(X(\C))\inj \KO(X'(\C))$. But, by Proposition \ref{PropRegSurNormalSontPoly}, we get $\KO(X'(\C)) = \C[X']$ so $\KO(X(\C)) \inj \C[X']$ and so $f\circ \piC \in  \C[X']$. By definition of $\C[X']$, we can consider a relation of the form $(f\circ\piC)^n+(a_{n-1}\circ\piC) (f\circ\piC)^{n-1}+...+(a_0\circ\piC) = 0$ on $X'(\C)$ with $a_0,...,a_{n-1} \in \C[X]$. Since $\piC$ is surjective, we get that $f^n+a_{n-1} f^{n-1}+...+a_0 = 0$ on $X(\C)$. Hence $f$ is integral on $\C[X]$.
\end{proof}

\begin{rmq}
In general, a function $f : X(\C) \to \C$ is integral on $\C[X]$ if and only if it is rational and locally bounded on $X(\C)$.
\end{rmq}

We obtain the following commutative diagram which sums up the situation : 
$$\xymatrix{
    \KO(X(\C)) \ar@{^{(}->}[r]^{\varphi} & \KO(X'(\C)) \\
    \C[X] \ar@{^{(}->}[u] \ar@{^{(}->}[r]& \C[X'] \ar@{=}[u]
}$$

As previously announced, we are going to give a description of the image of $\varphi$ in the case $\pi$ is finite.
\begin{prop}\label{PropConstanteSurLesFibres}
Let $\pi : Y \rightarrow X$ be a finite morphism of affine varieties. Then the image of $\varphi : \KO(X(\C)) \to \KO(Y(\C))$ is
$$\im(\varphi) = \{ f\in \KO(Y(\C))\text{ with $f$ constant on the fibers of }\piC \}$$
\end{prop}

\begin{proof}
Let $f \in \KO(Y(\C))$ be such that $$\forall y_1,y_2\in Y(\C) \text{ \hspace{0.5cm}} \piC(y_1) = \piC(y_2) \implies f(y_1) = f(y_2)$$
We consider
$$\begin{array}{cccccl}
g & : & X(\C)& \longrightarrow& \C&\\
&& x &\mapsto &f(y_i) &\text{with }y_i \in \piC^{-1}(\{x\})
\end{array}$$

The map is well defined by hypothesis and because $\pi$ is surjective.
Moreover we have $f = g \circ \piC$ and $f\in \Ccal^0(Y(\C),\C)$ so $g\circ \piC \in \Ccal^0(Y(\C),\C)$.\\
We now show that $g \in \Ccal^0(X(\C),\C)$. Let $F$ be a closed subset of $\C$. Then $(g\circ\piC)^{-1}(F) = \piC^{-1}(g^{-1}(F))$ is closed because $g\circ \piC =f$ is continuous and $\piC(\piC^{-1}(g^{-1}(F))) = g^{-1}(F)$ because $\piC$ is surjective. Thus, since $\piC$ is closed (see lemma \ref{lemPhiInj}), the set $g^{-1}(F)$ is closed and so $g$ is continuous.\smallskip \\
It remains to prove that $g$ is a rational function. We have $$\begin{array}{rll}
     g\circ\piC \in \KO(Y(\C)) & \implies g\circ\piC \text{ is integral on }\C[Y] &\text{ by Proposition \ref{PropRatiContEstEntiere}}  \\
     & \implies g\circ \piC \text{ is integral on }\C[X] &\text{ by \cite{AM} Corollary 5.4 p.60} \\
     & \implies g \text{ is integral on }\C[X] &\text{ because }\pi\text{ is surjective}\\
     & \implies g \in \KO(X(\C)) &\text{ by Proposition \ref{CoroClotureDansFoncCont} that we prove further away.}
\end{array}$$

This shows $f\in \im(\varphi)$. So we have proved that the set of functions in $\KO(Y(\C))$ which are constant on the fibers of $\piC$ is included in $\im(\varphi)$. The reverse inclusion being given by Lemma \ref{lemPhiInj}, we finally get
$$\im(\varphi) = \{ f\in \KO(Y(\C))\text{ with $f$ constant on the fibers of }\piC \}$$
\end{proof}

Now, considering the last proposition, it is natural to wonder what happens when there is only one element in every fiber of $\piC$. The answer is given in the following proposition.

\begin{prop}\label{PropSubEquiMemeFctRatioCont}
Let $\pi : Y \to X$ be a finite morphism of affine varieties. Then the following properties are equivalent :
\begin{enumerate}
    \item[1)] The extension $\pi^* : \C[X] \inj \C[Y]$ is subintegral.
    \item[2)] The morphism $\varphi: \KO(X(\C)) \to \KO(Y(\C))$ is an isomorphism.
    \item[3)] The morphism $\piC$ is a homeomorphim for the Euclidean topology.
\end{enumerate}
\end{prop}

\begin{proof}
Suppose $\C[X]\inj \C[Y]$ is subintegral, by Theorem \ref{TheoSubSsiBij} it means that $\piC : Y(\C)\to X(\C)$ is bijective. So, in this case, every function of $\KO(Y(\C))$ is clearly constant on the fibers of $\piC$. So, by Proposition \ref{PropConstanteSurLesFibres}, the map $\varphi$ is surjective and Lemma \ref{lemPhiInj} gives us the injectivity.\\

Conversely, if $\piC$ is not bijective, there exists $y_1 \neq y_2 \in Y(\C)$ such that $\pi(y_1) = \pi(y_2)$ and we can find $f \in \C[Y]$ such that $f(y_1) \neq f(y_2)$. Thus $f\in \KO(Y(\C))$ but $f \notin \im(\varphi)$.\\

Finally
$$\begin{array}{ccl}
\varphi : \KO(X(\C))\to \KO(Y(\C)) \text{ is an isomorphism} & \iff & \piC : Y(\C)\to X(\C) \text{ is bijective}\\
& \iff &\pi^* : \C[X] \inj \C[Y]\text{ subintegral }
\end{array}$$
The third statement comes from Lemma \ref{LemMorphFiniDoncFerme}
\end{proof}

Let $X$ be an affine variety, the coordinate ring of $X^+$ is the largest ring subintegral over $\C[X]$. Thus for every finite morphism between two varieties $\pi : Y\to X$ with $\C[X]\inj \C[Y]$ subintegral, we obtain the following commutative diagram :
$$\xymatrix{
    \KO(X(\C)) \ar[r]^{\simeq}& \KO(Y(\C)) \ar[r]^{\simeq} & \KO(X^+(\C)) \ar@{^{(}->}[r] & \KO(X'(\C))\\
    \C[X] \ar@{^{(}->}[u] \ar@{^{(}->}[r]^{subint.} & \C[Y] \ar@{^{(}->}[u] \ar@{^{(}->}[r]^{subint.} & \C[X^+] \ar@{^{(}->}[u] \ar@{^{(}->}[r] & \C[X'] \ar@{=}[u]
}$$
We will complete the diagram by showing $\C[X^+] = \KO(X^+(\C))$. We prove it in the next theorem by saying that the polynomial functions on the seminormalization are the polynomial functions on the normalization which are constant on the fibers of $\piC' : X'(\C) \to X(\C)$.\bigskip \\

\begin{theorem}\label{TheoRatioContEgalPolySN}
Let $X$ be an affine complex variety and $\pi^+ : X^+ \to X$ be the seminormalization morphism. We have the following isomorphism
$$\begin{array}{lccc}
     \varphi : & \KO(X(\C)) & \xrightarrow{\sim} & \C[X^+]\\
     & f & \mapsto & f\circ \piC^+ 
\end{array}$$
\end{theorem}

\begin{proof}
We have shown in Corollary \ref{CorSemiMaxEstSemi} that
$$\C[X^+] = \C[X^{+_{max}}] = \{ f\in \C[X']\text{ / }\forall x \in X(\C) \text{   }f_x\in \Ocal_{X,x} + \Rrm(\Ocal_{X',x}) \}$$
Let $\pi' : X' \to X$ be the normalization morphism of $X$. We want to show $$\C[X^{+_{max}}] = \{ f\in \C[X']\text{ / } f \text{ constant on the fibers of }\piC' \}$$
Let $x\in X(\C)$ and $f\in \C[X^{+_{max}}]$. We write $\piC'^{-1}(\{x\}) = \{ x_1',...,x_n' \}$. The goal is to show $\text{ }f(x_i')=f(x_j')$ for all $i,j \in \llbracket1,n\rrbracket$. First, the ideals of $\Ocal_{X',x}$ above $\m_x$ are of the from $\m_{x_i'} \Ocal_{X',x}$ :

$$\begin{array}{ccc}
\Ocal_{X,x} &\inj & \Ocal_{X',x} \smallskip \\
\m_x\Ocal_{X,x} &\mapsfrom& \m_{x_1'} \Ocal_{X',x}\\
&&\hspace{0.5cm}\vdots\\
&&\m_{x_n'} \Ocal_{X',x}
\end{array}$$

By definition, we have $f_x \in \Ocal_{X,x} + \Rrm(\Ocal_{X',x})$. So we can write $f_x = \alpha +\beta$ with $\alpha \in \Ocal_{X,x} \subset \Ocal_{X',x}$ and $\beta \in \m_{x_1'}\Ocal_{X',x} \cap ... \cap \m_{x_n'}\Ocal_{X',x}$. Thus, for all $i \in \{1,...,n\}$
$$f_x(x_i') = \alpha(\pi'(x_i')) + \beta(x_i')$$
But $\beta(x_i') = 0$ because $\beta\in \m_{x_i'}\Ocal_{X',x}$ and $\alpha(\pi'(x_i')) = \alpha(x)$. So $\alpha(x) = f(x_1') = ... = f(x_n')$ and we obtain 

$$\C[X^{+_{max}}] \subset \{ f\in \C[X']\text{ / } f \text{ constant on the fibers of }\piC' \}.$$

Conversely, let $f \in \C[X']$ be constant on the fibers of $\piC'$. Let $x\in X(\C)$, then $\forall y \in \piC'^{-1}(x)$, $f(y) = \alpha \in \C$. We then have $$f_x - \alpha \in \displaystyle{\bigcap_{y\in \piC'^{-1}(x)}\m_y\Ocal_{X',x}} = \Rrm(\Ocal_{X',x})$$ and so $f \in \C[X^{+_{max}}]$. We have proved $$\C[X^{+_{max}}] = \{ f\in \C[X']\text{ / } f \text{ constant on the fibers of }\piC' \}.$$

But since $\KO(X'(\C)) = \C[X']$ by Proposition \ref{PropRegSurNormalSontPoly} and since
$$\varphi : \KO(X(\C)) \xrightarrow{\sim} \{ f\in \KO(X'(\C))\text{ / } f \text{ constant on the fibers of }\piC' \}$$
by Proposition \ref{PropConstanteSurLesFibres}, we get
$$\varphi : \KO(X(\C)) \xrightarrow{\sim} \C[X^{+_{max}}] = \C[X^+].$$

\end{proof}

We have managed to identify the ring of continuous rational functions of an affine complex variety : it corresponds to the coordinate ring of its seminormalization. We can now complete the previous diagram.

For every morphism $\pi : Y\to X$ of affine varieties such that $\C[X]\inj \C[Y]$ is subintegral, we get
$$\xymatrix{
    \KO(X(\C)) \ar[r]^{\simeq}& \KO(Y(\C)) \ar[r]^{\simeq} & \KO(X^+(\C)) \ar@{^{(}->}[r] & \KO(X'(\C))\\
    \C[X] \ar@{^{(}->}[u] \ar@{^{(}->}[r]^{subint.} & \C[Y] \ar@{^{(}->}[u] \ar@{^{(}->}[r]^{subint.} & \C[X^+] \ar@{=}[u] \ar@{^{(}->}[r] & \C[X'] \ar@{=}[u]
}$$

\subsection{Continuous rational functions and regulous functions.}\label{SubsectionContinuousRationalvsRegulous}

As said before, continuous rational functions are of particular interest in real algebraic geometry. They are very related to an other kind of functions : the \textit{regulous} functions. Let $X$ be a real algebraic set and let $f : X\to \R$ be continuous. We say that $f$ is regulous on $X$ if for every algebraic subset $Z\subset X$, the restriction $f_{|_Z}$ has a rational representation. This is why they are sometimes called "hereditarily rational functions". Those two types of functions are not the same in the case of real singular algebraic sets. One can consider the following example (from \cite{Kollar}) of a continuous rational function which is not regulous.

\begin{ex}
Let $X := \{x^3 - (1+z^2)y^3\} \subset \R^3$ be an real algebraic set. Consider $f : X \to \R$ such that $f(x,y,z) = (1+z^2)^{\frac{1}{3}}$. See that, although $f$ is continuous on $X$ and $f = x/y$ if $y\neq 0$, the function $f$ is not rational on $\{y=0\}$.
\end{ex}

For more details, one can see \cite{Kuch} section 3 for a review on these notions. We show, in the following proposition, that these two kinds of functions are the same in complex algebraic geometry.

\begin{prop}\label{PropRatContSontRegulues}
Let $f \in \KO(X(\C))$. Then for every subvariety $V \subset X$, we have $$f_{|_{V(\C)}} \in \KO(V(\C))$$
\end{prop}

\begin{proof}
As usual, Lemma \ref{Lemirreducible} allows us to suppose $V$ irreducible. Let's considerate $\pi^+ : X^+ \to X$ the seminormalization morphism of $X$ and $V = \Vcal(\p)$ an irreducible subvariety of $X$. Like in Proposition \ref{PropSubEqBij}, we have the following commutative diagram :

$$\xymatrix{
    \C[X] \ar@{->>}[d] \ar@{^{(}->}[r]^{(\pi^+)^*}& \C[X^+] \ar@{->>}[d] \\
    \C[V] \ar@{^{(}->}[d] \ar@{^{(}->}[r]^{(\pi_{|W}^+)^*} &\C[W] \ar@{^{(}->}[d] \\
     \kappa(V) \ar@{=}[r] & \kappa(W),
}$$

with $W = \Vcal(\q)$ where $\q$ is the unique prime ideal of $\C[X^+]$ above $\p$. By the going-up property and the description of prime ideals for quotient rings, one can see that $W \to V$ is a bijection. Thus Theorem \ref{TheoSubSsiBij} tells us that $\C[V]\inj\C[W]$ is subintegral. Since $\C[W]$ is a finite $\C[V]$-module, we can apply Proposition \ref{PropSubEquiMemeFctRatioCont} and get that $\KO(V) \inj \KO(W)$ is an isomorphism. Thus
$$ f\in \KO(X(\C)) \implies f\circ\piC^+ \in \Ocal_{X^+}(X^+(\C)) \implies f\circ\pi^+_{|W(\C)} \in \Ocal_{X^+}(W(\C))$$
$$\implies f_{|V(\C)}\circ \pi_{|W(\C)}^+ = f\circ \pi^+_{|W(\C)}\in \KO(W(\C)) \implies f_{|V(\C)}\in \KO(V(\C))$$
\end{proof}

\begin{rmq}
In real algebraic geometry, regulous functions can also be defined in the following equivalent way. Let $f : X \to \R$ be a continuous function on a real algebraic set. We say that $f$ is \textit{regulous} ( or sometimes \textit{stratified-regular} ) if there exists a finite stratification $\mathcal{S}$ of $X$, with Zariski locally closed strata (i.e. the intersection of a closed and an open set), such that for all $S\in \mathcal{S}$ the restriction $f_{|_{S}}$ is regular. It also applies in our case, if $f : X(\C) \to \C$ is a continuous rational function, then we can write

$$ f = \left\{ \begin{array}{lll}
     & p_1/q_1 & \text{ if } q_1\neq 0  \\
     & p_2/q_2 & \text{ if } q_1 = 0\text{ and }q_2\neq 0  \\
     & p_3/q_3 & \text{ if } q_1=q_2=0\text{ and }q_3\neq 0  \\
     & ... &
\end{array} \right.$$ 
and for every $n > 1$, we have $\Zcal(q_n) \subset \sing(\Zcal(q_{n-1}))$.
\end{rmq}

\subsection{The ring of continuous rational functions seen as an integral closure and algebraic Whitney theorem.}\label{TheRingOfRat}

In Whitney's book \cite{Whitney}, one can find a chapter dedicated to a certain type of functions : the "c-holomorphic" functions. The c-holomorphic functions are defined as continuous functions on an analytic variety which are holomorphic on the smooth points of the variety. Note that, by Proposition \ref{PropRatEstRegSurRegX}, continuous rational functions are c-holomorphic. A characterization of c-holomorphic functions, given by Whitney, is that a continuous function is c-holomorphic if and only if it has an analytic graph. This theorem naturally leads to wonder if we can have the same characterization for continuous rational functions. In other words, do we have, on an affine algebraic variety, that a continuous function is rational if and only if its graph is algebraically closed ? The answer is yes and a proof with arguments from analytic geometry can be found in \cite{BDTW}. The goal of this section is to prove a slightly stronger version with arguments from algebraic geometry.\\
More precisely we aim to show that, if $X$ is an affine variety, then every continuous function from $X(\C)$ to $\C$, for which there exists $P(t)\in \C[X][t]$ such that $P(f)=0$, is rational.\\
It allows us to deduce the algebraic version of Whitney's theorem discussed above but also to identify the ring in which $\KO(X(\C))$ is the integral closure of $\C[X]$. 

We start by proving the theorem in the case where $X$ is irreducible and where the polynomial, for which the continuous function is a root, is irreducible in $\K(X)[t]$. With those hypothesis we can give a proof similar to the one given in \cite{Shafa2} (theorem 8.4 p.176). It is very important for the polynomial to be irreducible in $\K(X)[t]$ otherwise the new variety created from it won't necessarily be irreducible, whereas the key argument uses the irreducibility of this new variety.  

\begin{notation}
Let $P$ be a polynomial, we note $\disc(P)$ its discriminant.
\end{notation}

\begin{lemma}\label{LemEntContDonneRatIrréductible}
Let $X$ be an irreducible affine variety and $f: X(\C)\to \C$ be a continuous function. Suppose there exists an irreducible polynomial $P\in \K(X)[t]$ such that $$\exists U \text{ a Z-open set \hspace{0.5cm}}\forall x\in U(\C) \text{ \hspace{0.5cm}}P(x,f(x))=0$$
then
$$f \in \KO(X(\C))$$
\end{lemma}

\begin{proof}
First we consider the affine Z-open set $X_1$ such that $P$ is a monic polynomial of $\C[X_1][t]$. Then we write $Y_1 = \spec(\C[X_1][t]/<P>)$, which is irreducible because $P$ is irreducible in $\K(X)[t] = \K(X_1)[t]$, and $\pi : Y_1 \to X_1$ the induced finite morphism. We note $X_2$ the affine Z-open set where $\disc(P)$ does not vanish. Finally we write $Y_2 = \pi^{-1}(X_2)$. Now $X_2$ and $Y_2$ are two irreducible affine varieties with $\pi : Y_2 \to X_2$ finite and $$\forall x\in X_2(\C) \hspace{0.3cm} \#\piC^{-1}(x) = [\K(Y_2):\K(X_2)] = \deg(P).$$ We write $m:= \deg(P)$ and we prove by contradiction that $m=1$. Let's suppose $m>1$.

Let $x\in X_2(\C)$, we can consider $U_x$ a Euclidean open set such that $X_2(\C)\cap U_x$ is connected and $$\pi^{-1}(X_2(\C)\cap U_x) = \bigsqcup_{i=1}^m V_x^i\text{ where }V_x^i \subset Y_2(\C) \text{ are two by two disjoint connected open sets }$$
We note
$$\begin{array}{cccc}
     \varphi : & X_2(\C) & \to & Y_2(\C) \\
     &  x & \mapsto & (x;f(x)) 
\end{array} $$

which is, by hypothesis on $f$, a continuous section of $\pi$. Thus $\varphi(X_2(\C)\cap U_x)$ is connected and so it corresponds to one of the $V_x^i$ which we denote $V_x^{i_0}$. We then have that $\varphi(X_2(\C))$ and $Y_2(\mathbb{C})\backslash \varphi(X_2(\C))$ are open sets because $$\varphi(X_2(\C)) = \bigcup_{x\in X_2(\C)} V_x^{i_0}$$
and
$$ Y_2(\mathbb{C})\backslash \varphi(X_2(\C)) = \bigcup_{x\in X_2(\C)} \bigsqcup_{i\neq i_0} V_x^i $$
Moreover, since $m>1$, we clearly have $\varphi(X_2(\C)) \neq Y_2(\C)$. But since $Y_2$ is irreducible, the set $Y_2(\C)$ must be connected, a contradiction. So $m$ must be equal to 1 and then $f$ is rational.\\
\end{proof}

In order to prove the desired theorem with Lemma \ref{LemEntContDonneRatIrréductible}, we need to find an irreducible polynomial of $\K(X)[t]$ for which the continuous function is a root. This is what Lemma \ref{LemTrouverUnPolIrr} gives us.
\begin{lemma}\label{LemTrouverUnPolIrr}
Let $X$ be an affine irreducible smooth variety and $f:X(\C)\to \C$ be continuous with respect to the Euclidean topology. Suppose there exists a monic polynomial $P\in\C[X][t]$ such that $$\forall x\in X(\C)\text{ \hspace{0.5cm}} P(x,f(x))=0$$
Then $f$ is a root of an irreducible polynomial of $\K(X)[t]$.
\end{lemma}

\begin{proof}
Since $X$ is supposed to be smooth and $P$ to be a monic polynomial, we can apply \cite{Whitney} Lemma 2J Chapter 4 to get that $f$ is holomorphic on $X(\C)$. In particular, we get $f\in \mathcal{M}(X(\C))$ which is a field because $X$ is irreducible (see \cite{Shafa2} Theorem 7.1). Now consider the morphism
$$ \begin{array}{cccc}
     ev_f :& \K(X)[t] & \to & \mathcal{M}(X(\C))  \\
     & Q(t) & \mapsto & Q(f)
\end{array}$$
We have that $\K(X)[t]$ is a principal ideal ring because $\K(X)$ is a field. Since $P(f)=0$ in $\mathcal{M}(X(\C))$, then $\ker(ev_f)\neq 0$. So there exists $F\neq 0$ such that $\ker(ev_f) = <F>$. Since $\mathcal{M}(X(\C))$ is a domain, the polynomial $F$ is irreducible in $\K(X)[t]$.
\end{proof}
\newpage
We now have all the arguments we need to demonstrate the main theorem of this section.
\begin{theorem}\label{TheoEntContDonneRat}
Let $X$ be an affine variety and $f: X(\C)\to \C$ be a continuous function for the Euclidean topology. Suppose there exists a polynomial $P\in \C[X][t]$ which is non-zero on each irreducible component of $X(\C)$ and such that $$\forall x\in X(\C) \text{ \hspace{0.5cm}}P(x,f(x))=0$$
then
$$f \in \KO(X(\C))$$
\end{theorem}

\begin{proof}
Let $X = \bigcup_{i=1}^n X_i$ be its decomposition into irreducible components. Let $i\in \llbracket 1,n \rrbracket$ then $f_{|X_i(\C)}$ is a root of the polynomial $P$ with its coefficients restricted to $X_i(\C)$. Thus, by Lemma \ref{Lemirreducible}, it is enough to prove the theorem for an irreducible affine variety.\\
If $X_{sing}(\C) = \Zcal(<q_1,...,q_s>)$ and $a_n$ is the leading coefficient of $P$, we can replace $X$ by $\Dcal(q_1a_n)$ and then suppose that $X$ is smooth and that $P$ is a monic polynomial. It allows us to use Lemma \ref{LemTrouverUnPolIrr} and to get an irreducible polynomial $F \in \K(X)[t]$ such that there exists a Z-open set $U$ where for all $x\in U(\C)$, $F(x,f(x)) = 0$. The conclusion is now given by Lemma \ref{LemEntContDonneRatIrréductible}.
\end{proof}

Thanks to Theorem \ref{TheoEntContDonneRat} we can now see the ring of continuous rational functions as an integral closure of $\C[X]$.
\begin{cor}\label{CoroClotureDansFoncCont}
Let $X$ be an affine variety. Then
$$\KO(X(\C))  = \C[X]'_{\Ccal^0(X(\C),\C)}$$
\end{cor}

\begin{proof}
The result follows from Proposition \ref{PropRatiContEstEntiere} and Theorem $\ref{TheoEntContDonneRat}$.
\end{proof}

\begin{rmq}
By using Corollary \ref{CoroClotureDansFoncCont}, one can give a very short proof of Proposition \ref{PropRatContSontRegulues}. Indeed, if $V$ is a subvariety of $X$ and if $f \in \KO(X(\C))$, then there exists a monic polynomial in $\C[X][t]$ for which $f$ is a root. So $f_{|_{V(\C)}}$ is a root of the same polynomial with its coefficients restricted to $V(\C)$. Since $f_{|_{V(\C)}}$ is continuous, we get $f_{|_{V(\C)}} \in \C[V]'_{\Ccal^0(V(\C),\C)} = \KO(V(\C))$.
\end{rmq}

Let's conclude this section by proving the algebraic version of Whitney's Theorem 4.5Q in \cite{Whitney} introduced at the beginning of this section.

\begin{cor}\label{PropWhitney}
Let $X$ be an affine variety and $f:X(\C) \to \C$ be a continuous function. We note ${\Gamma_f :=\{ (x,f(x)) \mid x\in X(\C) \} \subset X(\C)\times \A^1(\C)}$ the graph of $f$. Then the following properties are equivalent :
\begin{enumerate}
    \item[1)] The graph $\Gamma_f$ is Z-closed.
    \item[2)] $f\in \KO(X(\C))$
\end{enumerate}
\end{cor}

\begin{proof}
The implication 1) implies 2) comes from Theorem \ref{TheoEntContDonneRat}. Conversely, we suppose $f\in \KO(X(\C))$ and we note $\pi : X^+ \to X$ the seminormalization morphism. By Proposition \ref{TheoRatioContEgalPolySN}, we have $f\circ \pi \in \C[X^+]$. Thus $$\Gamma_{f\circ \pi} = \{(y,f\circ\pi(y)) \mid y\in X^+(\C)\} $$ is Z-closed. Moreover the map $\pi\times Id$ is Z-closed because, by Theorem \ref{TheoSubSsiBij}, $\pi$ is a Z-homeomorphism. So $\pi\times Id(\Gamma_{f\circ\pi}) = \{ (\pi(y),f\circ\pi(y)) ; y\in X^+(\C) \}$ is Z-closed. Finally, since $\pi$ is bijective, we get
$$\left\{ (\pi(y),f\circ\pi(y)) \mid y\in X^+(\C) \right\} = \{ (x,f(x)) \mid x\in X(\C) \} = \Gamma_f$$
\end{proof}

\begin{rmq}
In \cite{BDT} and \cite{BDTW}, the authors consider c-holomorphic functions with an algebraic graph. Corollary \ref{PropWhitney} tells us that those functions are the same as the ones considered in this paper when we work on algebraic varieties.
\end{rmq}

\begin{rmq}
In real algebraic geometry, the zero sets of regulous functions are the closed sets of a thinner topology than the Zariski topology called the \textit{regulous topology}. In \cite{FHMM}, the authors show that we can recover some classical theorems of complex algebraic geometry if we work with the regulous topology instead of the Zariski topology. In our case, if $f\in \KO(X(\C))$ then Corollary 4.18 tells us that $\{x\in X(\C) \mid f(x)=0\} = \Gamma_f \cap (X(\C)\times \{0\})$ is a Zariski closed set.
\end{rmq}

\subsection{Examples of continuous rational functions.}\label{ExamplesOfRat}

In general, it is not easy to determine the seminormalization of a variety. We present in this subsection several examples of continuous rational functions and also some explicit seminormalizations of affine varieties. In order to do this we give a pleasant criterion to identify continuous rational functions.\\

\begin{theorem}\label{TheoGraphEntiereDoncCont}
Let $X$ be an affine variety and $f:X(\C)\to \C$. Then $f\in \KO(X(\C))$ if and only if it verifies the following properties :
\begin{center}
\begin{enumerate}
    \item[1)] $f\in \K(X)$
    \item[2)] There exists a monic polynomial $P(t)\in \C[X][t]$ such that $P(f)=0$ on $X(\C)$.
    \item[3)] The graph $\Gamma_f$ is Zariski closed in $X(\C)\times\mathbb{A}^1(\C)$.
\end{enumerate}
\end{center}

\end{theorem}

\begin{proof}
The direct implication is given by Propositions \ref{PropRatiContEstEntiere} and \ref{PropWhitney}. Conversely, suppose that $f$ verifies the three properties above. We consider the map
$$\begin{array}{cccc}
     \psi :&\C[X][t]&\to&\K(X)  \\
     &Q(t)& \mapsto & Q(f)
\end{array}$$
and write $\C[Y] \simeq \C[X][t]/\ker\psi \simeq \C[X][f]$ with $\pi:Y\to X$ the morphism induced by $\C[X]\inj \C[Y]$. We then have
$$\C[X] \inj \C[Y] \simeq \C[X][f] \subset \K(X)$$
So $\K(X)\simeq \K(Y)$ and $\pi$ is birational. Moreover $\C[Y]$ is a finite $\C[X]$-module because so is $\C[X][t]/<P(t)>$ and 
$$ \C[Y] \simeq \C[X][t]/\ker\psi \simeq ( \C[X][t]/<P(t)> ) / (\ker\psi / <P(t)> ) $$
Hence $\pi : Y \to X$ is a finite birational morphism. We want to show that $\piC$ is bijective. By hypothesis, there exists an ideal $I_f \subset \C[X][t]$ such that $\Gamma_f = \Zcal(I_f)$. We have $I_f \subset \ker\psi$ because $\forall Q\in I_f$ , $\forall x\in X(\C)$ ,  $Q(x,f(x))=0$. So 
$$ Y(\C) = \Zcal(\ker\psi) \subset \Zcal(I_f)=\Gamma_f = \{(x,f(x)) \mid x\in X(\C)\} $$
Then $$\forall x\in X(\C)\text{, \hspace{0.1cm}}\piC^{-1}(x) = \varnothing\text{ or }\{f(x)\}.$$
Since $\pi$ is finite, $\piC$ is surjective. So, for all $x\in X(\C)$, $\piC^{-1}(x)$ is not empty, which means that $\piC^{-1}(x) = \{f(x)\}$. We have shown $Y(\C)=\Gamma_f$ thus $\piC$ is bijective with the inverse map $x\mapsto (x;f(x))$. Thus $\piC$ is a finite birational and bijective morphism. From the universal property of the seminormalization, we get
$$ \C[X] \inj \C[Y] \inj \C[X^+] $$
that induces $$ X^+ \xrightarrow{\pi^+_Y} Y \xrightarrow{\pi} X$$
So, if we note $t\in \C[Y]$ such that $t:(x,f(x))\mapsto f(x)$, Theorem \ref{TheoRatioContEgalPolySN} gives us the existence of $g \in \KO(X(\C))$ such that $t \circ (\pi^+_Y)_{\C} = g \circ \piC \circ (\pi^+_Y)_{\C} $ so $t = g \circ \piC$.
Therefore, since $\piC$ is surjective, we get for all $x\in X(\C)$ 
$$g(x) = g\circ\piC(x;f(x)) = t(x;f(x)) = f(x)$$
Thus $f=g\in \KO(X(\C))$ which concludes the proof.
\end{proof}

\begin{ex}
Let $V = \spec(\C[x,y]/<y^2+(x^2-1)x^4=0>)$ and
$$\begin{array}{cccc}
     f: & V(\C) & \to & \C  \\
     & (x;y) & \mapsto & \left\{ \begin{array}{l} y/x\text{ if }x\neq 0 \\ 0\text{ else} \end{array}\right.
\end{array}$$ 
The function $f$ is a root of the polynomial $P_{(x;y)}(t) = t^2 + x^2(x^2-1)$. Since $P_{(0;0)}(t) = t^2$, all the values of $f$ are given by the roots of $xt-y =0 $ on $\{x\neq 0\}$ and else by the root of $P_{(0;0)}(t)$. Thus we have $\Gamma_f = \Zcal(<y^2+(x^2-1)x^4;xt-y;t^2 + x^2(x^2-1)>)$ which is a Zariski closed set. So Theorem \ref{TheoGraphEntiereDoncCont} tells us that $f\in \KO(V(\C))$.
\end{ex}

\begin{rmq}
The key thing in the criterion we gave is that $f$ is defined on all $V(\C)$. When one add rational functions to the coordinate ring of a variety to get its normalization, the functions are only defined on a Z-open set. Consider again the previous example but with the fraction $\frac{y}{x^2}\in \K(V)$. It is a root of $P_{(x;y)}(t) = t^2 + (x^2-1)$ on $\{x\neq 0\}$. But $P_{(0;0)} = t^2 -1$ as two distinct roots, so ${\Zcal(<y^2+(x^2-1)x^4;x^2t-y;t^2 + (x^2-1)>)}$ cannot be the graph of a map on $V(\C)$.\\
\end{rmq}

\begin{ex}
Now we give an example which illustrates the fact that a continuous rational function is a \textit{stratified-regular} function (see remark after Proposition \ref{PropRatContSontRegulues}). Let $V$ be a variety such that the set of its closed points, seen in $\A^4(\C)$, is defined by the following equations

$$V(\C) : \left\{ \begin{array}{lr}
     x^2+zyx+ty^2 = 0 & (1)  \\
     z^2+z^2t+t^3+yt = 0 & (2)\\ 
     t^2x^2+x^2y-y^2z^2 = 0 & (3)
     \end{array} \right.$$

Let $f:V(\C)\to \C$ such that

$$ f = \left\{ \begin{array}{lll}
     & x/y & \text{ if } y\neq 0  \\
     &z/t & \text{ if } y = 0\text{ and }t\neq 0  \\ 
     &0 & \text{ else}
\end{array} \right.$$ 

To show that $f$ is indeed a continuous rational function on $V(\C)$, we show that $f$ satisfies the three properties of Theorem \ref{TheoGraphEntiereDoncCont}. In particular we look at its graph $\Gamma_f$ and show that it is the following Z-closed set in $\A^5(\C)$ defined by

$$\Gamma_f = \left\{ \begin{array}{lr}
     x^2+zyx+ty^2 = 0 & (1)\\
     z^2+z^2t+t^3+yt = 0 & (2)\\ 
     t^2x^2+x^2y-y^2z^2 = 0 & (3)\\
     yX-x = 0 & (4)\\
     X^2 + zX +t = 0 & (5)\\
     t^2X^2 + xX - z^2 = 0 & (6)
     \end{array} \right.$$

First of all, let's verify that $f$ is indeed a root of the polynomials $(4),(5)$ and $(6)$ on $V(\C)$. We start by looking on $\Dcal(y)\subset V(\C)$ where $f=x/y$ :\\

$\begin{array}{lll}
     (4) & : & y\left(\frac{x}{y}\right)-x = 0 \\[0.3cm]
     (5) & : & \left(\frac{x}{y}\right)^2+z\left(\frac{x}{y}\right)+t = \frac{x^2+zyx+ty^2}{y^2} = 0\text{ by }(1)  \\ [0.3cm]
     (6) & : & t^2\left(\frac{x}{y}\right)^2+x\left(\frac{x}{y}\right)-z^2 = \frac{t^2x^2+x^2y-y^2z^2}{y^2} = 0\text{ by }(3)
\end{array}$\\

Now we check that it is still true on $\Zcal(y)\cap\Dcal(t)$ :\\

$\begin{array}{lll}
     (5) & : & \left(\frac{z}{t}\right)^2+z\left(\frac{z}{t}\right)+t = \frac{z^2+z^2t+t^3+yt}{t^2} = 0\text{ by }(2)  \\ [0.3cm]
     (6) & : & t^2\left(\frac{z}{t}\right)^2+x\left(\frac{z}{t}\right)-z^2 = \frac{t^2z^2+xzt-z^2t^2}{t^2} = 0\text{ since }y=0 \text{ implies }x=0\text{ by (1).}
\end{array}$\\

 We get that $f$ is a root of the polynomials $(4),(5)$ and $(6)$. It remains to see if the values of $f$ are completely determined by those polynomials.\\
 
If $y\neq 0$, then the equation $(4)$ forces the value of $f$ to be $x/y$ on $\Dcal(y)$. If $y = 0$ and $t\neq 0$, then the system $(4),(5),(6)$ becomes $\left\{ \begin{array}{l} X^2 + zX + t = 0  \\ X^2 = z^2 / t^2 \end{array} \right.$ which forces the value of $f$ to be $z/t$ on $\Zcal(y)\cap\Dcal(t)$. Finally if $y=t=0$, then the system becomes $X^2=0$.\\

We have shown that $\Gamma_f$ is completely described by the system given above. Thus $\Gamma_f$ is Z-closed. By (5) then $f$ is integral on $\C[X]$. By (4) then $f$ is rational on $V(\C)$. So Theorem \ref{TheoGraphEntiereDoncCont} tells us that $f \in \KO(V(\C))$.\\
\end{ex}

\begin{rmq}
The jacobian matrix of the equations defining $V$ is
$$ \text{Jac}(V) = \begin{pmatrix}
2x + yz & zx+2ty & xy & y^2\\
0 & t & 2z+2zt & z^2+y\\
2t^2x+2yx & x^2-2yz^2 & -2y^2z & 2tx^2
\end{pmatrix}$$ 

and, if $y=0$, it becomes :

$$ \text{Jac}(V)_{|_{\{y=0\}}} = \begin{pmatrix}
0&0&0&0\\
0 & t & 2z+2zt & z^2\\
0&0&0&0
\end{pmatrix}$$ 

So we have $\{y=0\} \subset V_{sing}(\C)$ which is coherent with Proposition \ref{PropRatEstRegSurRegX}.\\
\end{rmq}

\begin{rmq}
In the equations defining $\Gamma_f$, we could replace (6) by $(x-zt^2)X-(t^3+z^2)$.
\end{rmq}

\begin{ex}
It is shown in \cite{Dav} that, for plane curves, the seminormality can be read on the geometry of the singularities. A curve in $\A^2(\C)$ is seminormal if and only if its singularities are double points whose tangents are linearly independent. We illustrate this by looking at the example of three lines crossing at the origin in $\A^2(\C)$.

Let $V = \spec(\C[X;Y]/<XY(Y-X)>)$. It is clear that $V$ is not seminormal because the lines are not linearly independent. Let $f:V(\C)\to \C$ be such that
$$ f = \left\{ \begin{array}{lll}
     & \frac{2xy}{x+y} & \text{ if } (x;y)\neq (0;0)  \\
     &0 & \text{ else}
\end{array} \right.$$
We can see that $f$ is a root of the polynomial $P_{(x;y)}(t)=t^2-xy$ and that $\Gamma_f$ is equal to $\Zcal(xy(y-x);(x+y)t-2xy;t^2-xy)$ because $0$ is the only root of $P_{(0;0)}$. So, by Theorem \ref{TheoGraphEntiereDoncCont}, we have $f \in \KO(V(\C))$. Furthermore, we have $\Gamma_f = V^+(\C)$ because the graph corresponds to three linearly independent lines in $\A^3(\C)$. Indeed it is the union of three lines crossing at the origin with direction vectors (1,0,0), (0,1,0) and (1,1,1).\\
Another way to see that $f$ is continuous is that $f_{\mid x=0} = 0$, $f_{\mid y=0} = 0$ and $f_{\mid x=y} = x$. So $f$ is a continuous rational function on each irreducible component of $V(\C)$. Thus $f \in \KO(V(\C))$ by Lemma \ref{Lemirreducible}.
\end{ex}

\begin{rmq}
Let $X$ be an affine variety. Since $\C[X]\inj \C[X^+]$ is finite and $\C[X]$ is a noetherian ring, one can show that the process of adding elements $f_i \in \KO(X(\C))$ with $f_{i+1} \notin \C[X][f_1,...,f_i]$ ends after a finite number of steps. 
\end{rmq}

\subsection{Nullstellensatz for complex regulous functions.}

A very important property of the regulous functions in real algebraic geometry is the regulous version of the Nullstellensatz (\cite{FHMM} Theorem 5.24). We give here a regulous version of the Nullstellensatz for complex affine varieties. One can also find a proof of this result for c-holomorphic functions with algebraic graph in \cite{BDT}.

We consider the same notations as in Theorem \ref{TheoRatioContEgalPolySN}. So, if $X$ is an affine variety and $\pi : X^+ \to X$ is its seminormalization morphism, we consider the isomorphism $$\begin{array}{lccc}
     \varphi : & \KO(X(\C)) & \xrightarrow{\sim} & \C[X^+]\\
     & f & \mapsto & f\circ \piC 
\end{array}$$

Let $I \subset \KO(X(\C))$. We write
$$\Zcal^0(I) := \{ x\in X(\C) \mid \forall f\in I\text{, }f(x)=0 \}$$
Let $E \subset X(\C)$. We write $$ \Jcal^0(E) := \{ f\in \KO(X(\C)) \mid \forall x\in E\text{, } f(x) = 0\}$$
Let $I$ be an ideal of $\KO(X(\C))$, then see that $I$ is of the form $I = <g_1,...,g_n>$ by noetherianity of $\C[X^+]$. So $$\piC^{-1}(\Zcal^0(I)) = \piC^{-1}(\cap \Zcal^0(g_i)) = \cap \piC^{-1}(\Zcal^0(g_i)) = \cap \Zcal(g_i\circ \piC) = \Zcal(\varphi(I))$$

\begin{theorem}[Nullstellensatz]\label{TheoNullstellensatz}
Let $X$ be an affine complex variety and $I$ be an ideal of $\KO(X(\C))$. Then 
$$ \Jcal^0(\Zcal^0(I)) = \sqrt{I} $$
\end{theorem}

\begin{proof}
One of the inclusion is clear. For the other inclusion, we consider $f \in \Jcal^0(\Zcal^0(I))$. It is equivalent to say that $\Zcal^0(I) \subset \Zcal^0(f)$. Then $\Zcal(\varphi(I)) = \piC^{-1}(\Zcal^0(I)) \subset \piC^{-1}(\Zcal^0(f)) = \Zcal(f\circ \piC)$. Then, by the classical Nullstellensatz on $\C[X^+]$, we can consider $n\in \mathbb{N}$ such that $(f\circ \piC)^n \in \varphi(I)$. So we get $f^n = \varphi^{-1}(\varphi(f)^n) \in \varphi^{-1}(\varphi(I)) = I$ and finally $f\in \sqrt{I}$.
\end{proof}

We also get a version of the Nullstellensatz where we want to study only one element of $\KO(X(\C))$. We will need this result in Section \ref{SectionCritères}. One can do the exact same proof as Theorem \ref{TheoNullstellensatz} by adapting it with the following notations. For $f \in \KO(X(\C))$, $I \subset \C[X][f]$ and $E\subset X(\C)$, consider $\Zcal^f(I) := \{ x\in X(\C) \mid \forall g\in I\text{, }g(x)=0 \}$ and $\Jcal^f(E) := \{ g\in \C[X][f] \mid \forall x\in E\text{, } g(x) = 0\}$. Also, consider $\C[Y] \simeq \C[X][t]/I_f$, $\pi : Y\to X$ and $\varphi : \C[X][f] \xrightarrow{\sim} \C[Y]$.

\begin{theorem}\label{TheoNullstellensatzElement}
Let $X$ be an affine complex variety. Let $f \in \KO(X(\C))$ and $I$ be an ideal of $\C[X][f]$ then 
$$ \Jcal^f(\Zcal^f(I)) = \sqrt{I} $$
\end{theorem}

\section{Classical results on seminormality with regulous functions.}\label{SectionClassicalResults}

We revisit several results on seminormality using regulous functions. In this section $X$ will be an affine variety. If $f\in \KO(X(\C))$, then we have shown in the previous section that $\Gamma_f$, the graph of $f$, is a Z-closed set of $X(\C)\times\A^1(\C)$. So there exists an ideal $I_f \subset \C[X][t]$ such that $\Gamma_f = \Zcal(I_f)$. Moreover, we have $\C[X][t]/I_f \simeq \C[X][f]$ and, since $f$ is integral over $\C[X]$, the ring $\C[X][f]$ is a $\C[X]$-module of finite type. We note $\Cond(f) := (\C[X] : \C[X][f]) = {\{ p\in \C[X] \mid p.\C[X][f] \subseteq \C[X] \}}$ the conductor of $\C[X]$ in $\C[X][f]$.\\

\subsection{Definitions and criteria of seminormality in commutative algebra.}\label{SectionCritères}

In this paper we have used Traverso's definition of the seminormalization \cite{T} where, for an integral extension of rings $A\inj B$, the seminormalization of $A$ in $B$ is given by \[A^+_B = \{b\in B \mid \forall \p \in \spec(A)\text{, }b_{\p}\in A_{\p}+\Rrm(B_{\p}) \}\] But, as explain in \cite{V}, there are several definitions of the seminormalization for commutatives rings. For Hamann a ring $A$ is seminormal in $B$ if, for $n\in\mathbb{N}^*$, $A$ contains all the elements $b\in B$ such that $b^n,b^{n+1}\in A$. The equivalent definition used by Leahy and Vitulli consist in replacing $n$ and $n+1$ by any positive relatively prime integers. Finally Swan gave an other definition of the seminormalization which is not equivalent to the previous ones for general commutative rings. Our goal in this section is to reinterpret those definitions in term of regulous functions and to see that they are all equivalent for affine rings.

\begin{definition}\label{DefElementaryExtension}
Let $A\inj B$ be an extension of rings and $b\in B$ be such that $b^2$, $b^3\in\C[X]$. In this case, we say that $A\inj A[b]$ is an elementary subintegral extension.
\end{definition}

It is shown in \cite{Swan} that, if a ring $A$ is not seminormal in an other ring $B$, then we can always find a proper elementary subintegral subextension of $A\inj B$. The following proposition gives a similar result with regulous functions.

\begin{prop}\label{PropSwan}
Let $X$ be a complex affine variety and $f \in \KO(X(\C))\setminus \C[X]$. Then there exists an element $g\in \C[X][f]\setminus \C[X]$ such that $g^n \in \Cond(f) \subset \C[X]$, for all integer $n\geqslant 2$.
\end{prop}

\begin{proof}
We know by Proposition \ref{PropRatContSontRegulues} that $f$ can be writen in the following way
$$f = \left\{ \begin{array}{cll}
    p_1/q_1&\text{ if }&q_1\neq 0\\
    p_2/q_2&\text{ if }&q_1=0\text{ and }q_2\neq 0\\
    \vdots\\
    p_{n-1}/q_{n-2}&\text{ if }&q_1=...=q_{n-2}=0\text{ and }q_{n-1}\neq 0\\
    p_n&\text{ if }&q = q_1 = ... = q_{n-1} = 0
\end{array} \right.$$
We consider the minimal integer $s$ such that $q_{s+1}f \notin \C[X]$. If $s$ exists, we continue the proof with $q_{s+1}f-p_{s+1} \notin \C[X]$. If $s$ doesn't exists, we take $f - p_n$. So we can suppose that
$$f = \left\{ \begin{array}{cll}
    p_1/q_1&\text{ if }&q_1\neq 0\\
    p_2/q_2&\text{ if }&q_1=0\text{ and }q_2\neq 0\\
    \vdots\\
    p_s/q_s&\text{ if }&q_1=...=q_{s-1}=0\text{ and }q_s\neq 0\\
    0&\text{ if }&q = q_1 = ... = q_s = 0
\end{array} \right.$$
with $q_if\in\C[X]$ and so $q_i\in\sqrt{\Cond(f)}$ for all $i\leqslant s$. Let's consider $I = <q_1^{n_1},...,q_s^{n_s}>$ with $n_i\in \mathbb{N}$ such that $q_i^{n_i} \in \Cond(f)$. See that $\Zcal^f(I) \subset \Zcal^f(f)$. So, by Theorem \ref{TheoNullstellensatzElement}, we have $f\in \sqrt{I}$. Since $I\subset \Cond(f)$, we get $f\in \sqrt{\Cond(f)}$. So we can consider the minimal integer $m\geqslant 1$ such that $f^m\notin \Cond(f)$ and $f^{m+1} \in \Cond(f)$. The fact that $f^m\notin \Cond(f)$ means that there exists $h = a_0 + a_1.f+...+a_d f^d \in \C[X][f]$, where $d:=\deg(f)-1$, such that $f^m.h = a_0.f^m + a_1.f^{m+1}+...+a_{m+d} f^{m+d} \notin \C[X]$. But since $f^{m+1} \in \Cond(f)$, we get $a_1.f^{m+1}+...+a_{m+d} f^{m+d} \in \C[X]$. It implies that $a_0f^m\notin \C[X]$ and so $f^m\notin \C[X]$. Finally, we write $g := f^m$ and we have find an element $g\in \C[X][f]\setminus\C[X]$ such that $g^n\in\Cond(f) \subset \C[X]$, for all integer $n\geqslant2$.
\end{proof}

We recover now, with regulous functions, that Traverso, Hamann and Leahy-Vitulli's definitions of the seminormalization are equivalent. In order to do this, we show the following Lemma.

\begin{lemma}\label{LemNMCLOSED}
Let $f\in K(X)$ such that there exists $n,m\in\mathbb{N}^*$ with $\gcd(n,m)=1$ and $f^n,f^m\in\C[X]$. Consider $u,v\in\mathbb{Z}$ such that $un+vm=1$ and assume that $u>0$ and $v<0$. Then $$g=\left\{ \begin{array}{l}
    f \text{ if }f^m\neq 0\\
    0 \text{ else}
\end{array} \right. \in \KO(X(\C))$$
\end{lemma}

\begin{proof}
Let $X_i$ be an irreducible component of $X$. If $f_{\mid_{X_i(\C)}}^m = 0$, then $f_{\mid_{X_i(\C)}} = 0$ so we define $g_{\mid_{X_i(\C)}} = 0$. So, by Lemma \ref{Lemirreducible}, we can suppose that $X$ is irreducible and that $f^m\neq 0$. We have $$g=\left\{ \begin{array}{l}
    f \text{ if }f^m\neq 0\\
    0 \text{ else}
\end{array} \right. = \left\{ \begin{array}{cl}
    (f^n)^u/(f^m)^{-v} & \text{ if }f^m\neq 0\\
    0 & \text{ else}
\end{array} \right.$$
So the graph of $g$ is $\Gamma_g = \Zcal((f^m)^{-v}t-(f^n)^u;t^m-f^m)$ and, by Theorem \ref{TheoGraphEntiereDoncCont}, we get $g\in \KO(X(\C))$.
\end{proof}

The Lemma tells us that a fraction with one of the property appearing in the following criteria extend into a regulous function. So if $X$ is seminormal it has to contain the elements mentioned in criteria 3),4) and 5). Moreover Proposition 5.2 shows that the seminormalization is the reunion of all of this kind of elements. So it is sufficient for $X$ to contain those elements in order to be seminormal. This is how we obtain the following Proposition.

\begin{prop}[Hamann and Leahy-Vitulli's criteria]\label{PropVitulliCriterion}
Let $X$ be an affine complex variety. Then the following statements are equivalent :
 \begin{enumerate}
    \item $X$ is seminormal.
    \item $ \forall f\in \C[X']$ the conductor of $\C[X]$ in $\C[X][f]$ is a radical ideal of $\C[X][f]$.
    \item $ \forall f\in \K(X)\text{ \hspace{0.1cm}}f^2,f^3\in \C[X] \implies f\in \C[X]$.
    \item $\forall f\in \K(X)\text{ \hspace{0.1cm}}f^n,f^m\in \C[X] \implies f\in \C[X]$, for some $m,n\in \mathbb{N}$ relatively prime.
    \item $\forall f\in \K(X)\text{ \hspace{0.1cm}}f^n,f^{n+1}\in \C[X] \implies f\in \C[X]$, for some $n\in \mathbb{N}$.
\end{enumerate}
\end{prop}

\begin{proof}
$2) \implies 1)$. If $X$ is not seminormal, then there exists $f\in \KO(X(\C))\setminus \C[X]$. So Proposition $\ref{PropSwan}$ gives an element $g\in \C[X][f]$ such that $g$ belongs to the radical of $(\C[X] : \C[X][f])$ but not to the conductor itself. The fact that $g\notin \C[X]$ and $g^n\in\C[X]$ for all $n\geqslant2$, shows that $3),4)$ or $5) \implies 1)$.\\ 

$1) \implies 2)$. Suppose there exists $f\in \K(X)$ and $g\in \sqrt{\Cond(f)}\setminus\Cond(f)$. We can consider $n\in \mathbb{N}^*$ such that $g^{n-1}\notin\Cond(f)$ and $g^n\in\Cond(f)$. So there exists $h\in\C[X][f]$ such that $g^{n-1}h\notin \C[X]$ and $(g^{n-1}h)^2$, $(g^{n-1}h)^3 \in \C[X]$. Then, by Proposition \ref{LemNMCLOSED}, we get 
$$\psi=\left\{ \begin{array}{cl}
    g^{n-1}h & \text{ if }(g^{n-1}h)^2\neq 0\\
    0 & \text{ else}
\end{array} \right. \in \KO(X(\C))$$
So $\psi \in \KO(X(\C))\setminus \C[X]$ which means that $X$ is not seminormal.\\


$1)\implies 4)$. Take $n,m\in \mathbb{N}$ such that $\gcd(n;m)=1$ and assume $X$ is seminormal. Consider $f\in\K(X)$ such that $f^n,f^m\in \C[X]$. Then, by Proposition \ref{LemNMCLOSED}, we can extend $f$ to a regulous function. So we get $f \in \KO(X(\C)) = \C[X]$. Since, for all $n\in \mathbb{N}$, we have $\gcd(n;n+1)=1$, we also get $1)\implies 3)$ and $5)$.
\end{proof}

We recover now that Traverso and Swan's definitions of the seminormalization are equivalent for affine rings by using regulous functions. First we get the following Proposition which gives us a way to construct regulous functions from polynomials that respect a certain type of relation.

\begin{prop}\label{PropConstructRegulues}
Let $p,q\in\C[X]$ be such that there exists $n\in \mathbb{N}^*$ with $p^n \in <q^{n+1}>$. Then $$f = \left\{ \begin{array}{cl}
    p/q & \text{ if }q\neq 0\\
    0 & \text{ else}
\end{array} \right. \in \KO(X(\C))$$
\end{prop}

\begin{proof}
Consider $n\in\mathbb{N}^*$ such that $p^n \in <q^{n+1}>$. Then there exists $h\in\C[X]$ such that $p^n = h q^{n+1}$. So, if $X_i$ is an irreducible component of $X$ such that $q=0$, we get that $p=0$ and so we define $f_{\mid_{X_i(\C)}} = 0$. Then, by Lemma \ref{Lemirreducible}, we can suppose $X$ irreducible and $q\neq0$. In this case, the graph of $f$ is given by $\Gamma_f = \Zcal(I_X;qt-p ; t^n-qh)$ and we can apply Theorem \ref{TheoGraphEntiereDoncCont} to conclude.
\end{proof}
\newpage
\begin{lemma}\label{LemSwanEtVitulli}
Let $X$ be an affine variety and $p,q \in \C[X]$. We write $$f = \left\{ \begin{array}{l}
    p/q \text{ if }q\neq 0\\
    0 \text{ else}
\end{array} \right.$$\\
Then $$p^2 = q^3\text{ if and only if }f^2 = q\text{ and }f^3 = p$$
In this case $f\in \KO(X(\C))$ and $$\Gamma_f = \Zcal(I_X;qt-p;t^2-q) = \Zcal(I_X;t^2-q;t^3-p)$$
\end{lemma}

\begin{proof}
Let $X_i$ be an irreducible component of $X$ such that $q=0$. Then $f_{\mid_{X_i(\C)}} = 0$ and the lemma becomes trivial. So, by Lemma \ref{Lemirreducible}, we can suppose $X$ irreducible with $q\neq 0$. In this case, if $p^2=q^3$ then $f^2 = p^2/q^2 = q^3/q^2 = q$ and $f^3 = p^3/q^3 = p^3/p^2 = p$ if $p,q \neq 0$. Moreover, if $q=0$, then $f^2 = q = f^3 = p = 0$. So $f^2=q$ and $f^3 = p$ on $X(\C)$. Conversely, if $f^2=q$ and $f^3=p$ then $p^2 = (f^3)^2 = (f^2)^3 = q^3$. We get that $f\in\KO(X(\C))$ by Proposition \ref{PropConstructRegulues}.
\end{proof}

The Lemma shows that the relations of the form $p^2 = q^3$ produce regulous functions and Proposition \ref{PropSwan} tells us that the seminormalization is the reunion of all of this kind of functions. Hence we obtain Swan's criterion.

\begin{prop}[Swan's criterion]
Let $X$ be an affine complex variety. Then the following statements are equivalent :
\begin{enumerate}
    \item $X$ is seminormal.
    \item For all $p,q\in \C[X]$ such that $p^2=q^3$ there exists $f\in \C[X]$ with $f^2=q$ and $f^3=p$.
\end{enumerate}
\end{prop}

\begin{proof}
$1)\implies 2)$. Suppose $X$ is seminormal and let $p,q\in\C[X]$ with $p^2=q^3$. Then by Lemma \ref{LemSwanEtVitulli} we get an element $f\in \KO(X(\C))$ such that $f^2=q$ and $f^3=p$. Since $X$ is seminormal, we have $f\in \C[X]$.\\

$2)\implies 1)$. Suppose that $X$ is not seminormal, then Proposition \ref{PropSwan} gives us an element $g\in \KO(X(\C))\backslash\C[X]$ with $g^2,g^3\in \C[X]$. So if we write $q:=g^2$ and $p:=g^3$, Lemma \ref{LemSwanEtVitulli} tells us that $p^2 = q^3$. Thus, if there exists $f\in \C[X]$ with $f^2=q$ and $f^3=p$, we get $f=g$ on $\Dcal(q)$. By continuity, we get $f = g$ on $X(\C)$ which is impossible because $g\notin \C[X]$.
\end{proof}

\subsection{Localization and seminormalization.}

It is shown, for general rings, that the operation of localization and seminormalization commute. In Traverso \cite{T}, it is proved by considering special subextensions between the seminormalization and the normalization of the ring. In Swan \cite{Swan}, it is proved by considering elementary subintegral extensions of the ring (see Definition \ref{DefElementaryExtension}). We propose here, because we will need it in Proposition \ref{PropLoja}, a proof with regulous functions but only in the case of the localization by a single element because we need $S^{-1}\C[X]$ to be affine.\\
\newpage
\begin{prop}[Seminormalization and localization by a single element]\label{PropLocalisation}
Let $X$ be a complex affine variety and $S$ be a multiplicative set of $\C[X]$ such that $S=\{1,q,q^2,...\}$ with $q\in \C[X]$
. Then
$$ S^{-1}\C[X^+]=(S^{-1}\C[X])^+ $$
\end{prop}

\begin{proof}
First, see that it is equivalent to show $$S^{-1}\KO(X(\C)) = \KO(\Dcal(q))$$
The inclusion $S^{-1}\KO(X(\C)) \subset \KO(\Dcal(q))$ is clear because if $f\in \KO(X(\C))$, then for all $s\in S$ the function $f/s$ is still rational and continuous on $\Dcal(q)$. To get the other inclusion, we must show
$$\forall g\in \KO(\Dcal(q)) \text{ \hspace{0.5cm}} \exists s\in S \text{ \hspace{1cm}} sg = \left\{ \begin{array}{ll}
    s(x)g(x) & \text{ if }x\in \Dcal(q)\\
    0 & \text{ else }
\end{array} \right.\in\KO(X(\C))$$

So let $g\in \KO(\Dcal(q))$. Then, by Theorem \ref{TheoGraphEntiereDoncCont}, it verifies the three following properties :
\begin{enumerate}
    \item $g\in \K(\Dcal(q))$
    \item $g$ is the root of a monic polynomial whose coefficients are in $S^{-1}\C[X]$.
    \item The graph $\Gamma_g \subset \Dcal(q)\times\A^1(\C)$ of $g$ is Z-closed.
\end{enumerate}
In other words, $g\in \K(\Dcal(q))$ and there exists a system of polynomials 
$$ (*) : \left\{ \begin{array}{l}
      P(x,t) = t^d + \frac{a_{d-1}}{s_{d-1}}.t^{d-1} + ... + \frac{a_0}{s_0} = 0 \\ [0.2cm]
      F_1(x,t) = \frac{a_{1,d_1}}{s_{1,d_1}} .t^{d_1} + \frac{a_{1,d_1-1}}{s_{1,d_1-1}} .t^{d_1-1} + ... + \frac{a_{1,0}}{s_{1,0}} = 0\\ [0.1cm]
      \vdots\\ [0.1cm]
      F_n(x,t) = \frac{a_{n,d_n}}{s_{n,d_n}} .t^{d_n} + \frac{a_{n,d_n-1}}{s_{n,d_n-1}} .t^{d_n-1} + ... + \frac{a_{n,0}}{s_{n,0}} = 0
\end{array} \right.$$
such that, for all $x\in \Dcal(q)$, the element $g(x)$ is its only solution. We have to see if there exists an $s \in S$ such that $sg$ verifies a system of the similar form on $\C[X]$. Lets consider
$$ s = \left( \prod_{k=0}^{d-1} s_k \right)^2 . \left( \prod_{i,j\in\llbracket 1;n \rrbracket\times \llbracket 0;d_i \rrbracket } s_{i,j} \right)^2$$ 
We show that $sg$ is the only solution of the following system whose coefficients are in $\C[X]$ :
$$ (**) : \left\{ \begin{array}{l}
      \tilde{P}(x,t) = t^d + \frac{s}{s_{d-1}}a_{d-1}.t^{d-1} + ... + \frac{s}{s_0}a_0.s^{d-1} = 0 \\[0.4cm]
      \tilde{F_1}(x,t) = \frac{s}{s_{1,d_1}}a_{1,d_1}.t^{d_1} + \frac{s}{s_{1,d_1-1}}a_{1,d_1-1}.s.t^{d_1-1} + ... + \frac{s}{s_{1,0}}a_{1,0}.s^{d_1} = 0\\[0.2cm]
      \vdots\\[0.2cm]
      \tilde{F_n}(x,t) = \frac{s}{s_{n,d_n}}a_{n,d_n}.t^{d_n} + \frac{s}{s_{n,d_n-1}}a_{n,d_n-1}.s.t^{d_n-1} + ... + \frac{s}{s_{n,0}}a_{n,0}.s^{d_n} = 0
\end{array} \right.$$
Indeed, for all $x\in \Dcal(q)$, we have
$$ \left\{ \begin{array}{l}
      \tilde{P}(x,sg(x))=s^d(x)P(x,g(x))=0\\[0.2cm]
      \tilde{F_i}(x,sg(x)) = s^{d_i+1}(x)F_i(x,g(x))=0\\
\end{array} \right.$$
So $sg(x)$ is the only solution of the system $(**)$ for all $x\in \Dcal(q)$. Now see that, in the definition of $s$, we carefully took squared elements so that if $x\in \Zcal(s)=\Zcal(q)$ then all the coefficients in (**) vanish except $t^d$ in $P(x,t)$. Thus, for all $x\notin \Dcal(q)$, we get $t=0$ and so $sg(x)$ is the only solution of the system $(**)$ for all $x\in X(\C)$.
\end{proof}

The fact that localization and seminormalization commute leads to look at seminormality directly at the points of a variety.

\begin{definition}
Let $X$ be an affine variety. We define the set of seminormal points in $X$ by
$$\SN(X) := \{ x\in X \mid \Ocal_{X,x} \text{ is seminormal} \}$$
and the seminormal points of $X(\C)$ by $\SN(X(\C)) = \SN(X)\cap X(\C)$.
\end{definition}

Now we can improve proposition \ref{PropRegSurPtsNormaux} and be more precise about the points where regulous functions are regular.

\begin{prop}\label{PropRatSontRegSurPointsSN}
Let $X$ be an affine variety and $f \in \KO(X(\C))$. Then $f\in \Ocal_{X,x}$ for all $x\in \SN(X(\C))$.
\end{prop}

\begin{proof}
We prove $1) \implies 2)$. Let $f\in \KO(X(\C))$ and $x\in \SN(X(\C))$, we consider $\pi^+ : X^+ \to X$ the seminormalization morphism of $X$ and $x^+ \in X^+(\C)$ such that $\pi^+(x^+) = x$. By Theorem \ref{TheoRatioContEgalPolySN}, we have $f\circ \pi^+_{\C} \in \C[X^+]$. Since seminormalization and localization commute, we have $f\circ \piC^+ \in \Ocal_{X^+,x^+} = \Ocal_{X,x}^+$ and since $x\in \SN(X(\C))$, we have that $\Ocal_{X,x} \inj \Ocal_{X,x}^+$ is an isomorphism. So $f\in \Ocal_{X,x}$.
\end{proof}

\begin{rmq}
An element in $\bigcap_{x\in \SN(X(\C))} \Ocal_{X,x}$ does not always extends by continuity. One can take the example $X = \spec(\C[x,y]/<y^2+(x^2-1)x^4>)$ given after Theorem \ref{TheoGraphEntiereDoncCont}. We have $\SN(X(\C)) = \{ (0;0) \} = \{x=0\}$ but the fraction $\frac{y}{x^2}$ cannot be continuously extended on $X(\C)$.
\end{rmq}

 We can deduce from Proposition \ref{PropRatSontRegSurPointsSN} a classical result about seminormalization.
 \begin{cor}
 Let $X$ be an affine variety. Then $X$ is seminormal if and only if $\SN(X(\C)) = X(\C)$.
 \end{cor}

\begin{proof}
Suppose that $X$ is seminormal, then $\C[X^+] = \C[X]$. So it is clear that $\Ocal_{X,x} = \Ocal_{X^+,x^+}$ for all $x\in X(\C)$. Conversely, suppose that $\SN(X(\C)) = X(\C)$ and let $f \in \KO(X(\C))$. Then, by Proposition \ref{PropRatSontRegSurPointsSN}, we have $f\in \bigcap_{x\in X(\C)} \Ocal_{X,x} = \C[X]$. So $\KO(X(\C)) = \C[X]$ and $X$ is seminormal.
\end{proof}

\section{The sheaf of complex regulous functions.}\label{SectionTheSheafOf}

In the paper \cite{FHMM} introducing the regulous functions on real algebraic varieties, the authors define the sheaf of regulous functions because they wanted to recover some classical theorems from complex algebraic geometry for real varieties equipped with the sheaf of regulous functions. In the same spirit, we look at the sheaf of complex regulous functions and we will notably show that, if $X$ is an affine variety, then $(X,\KO_X)$ is isomorphic to $(X^+,\Ocal_{X^+})$. In particular $(X,\KO_X)$ is an affine scheme.\\

\begin{definition}
Let $X$ be an affine variety and $U$ be a Z-open set of $X(\C)$. Then we write $\KO_X(U)$ the set of continuous functions $f:U\to \C$ for the Euclidean topology which are regular on a Z-open Z-dense subset of $U$.
\end{definition}

\newpage
By doing the exact same proof as that of Lemma \ref{Lemirreducible}, we get the following result.
\begin{lemma}\label{LemIrreducibleLocal}
Let $X$ be an affine variety, $U$ be a Z-open set of $X(\C)$ and $f$ be an element of $\KO_X(U)$. We note $\{X_i\}_{i\in \llbracket1;n\rrbracket}$ the irreducible components of $X$. Then the following statements are equivalent :
\begin{enumerate}
    \item[1)] $f\in \KO_X(U)$
    \item[2)] $\forall i\in \llbracket 1;n \rrbracket \hspace{0.3cm} f_{|U\cap X_i(\C)} \in \KO_X(U\cap X_i(\C))$
\end{enumerate}
\end{lemma}

\begin{rmq}
The sets of the form $\Zcal^0(I)$ where $I$ is an ideal of $\KO(X(\C))$, which were introduced in Theorem \ref{TheoNullstellensatz}, define the same topology as the Zariski topology. Despite this fact, some differences may occur. For example, the sets of the form $\Dcal^0(f)$ defined below might only be affine on the seminormalization of the variety.
\end{rmq}

\begin{definition}
Let $X$ be an affine variety. We define the \textit{regulous topology} on $X(\C)$ to be the topology whose open sets are generated by the sets
$$ \Dcal^0(f) := \{x\in X(\C) \mid f(x)\neq 0\} $$
where $f$ is an element of $\KO(X(\C))$.
\end{definition}

Now that we have a local definition for regulous functions, we define the sheaf $\KO_X$.
\begin{prop}
Let $X$ be an affine variety. The presheaf defined by  $$ \begin{array}{cccc}
    \KO_X : & \{\text{ regulous open sets of } X(\C) \}^{\text{op}} & \to & \mathbf{Ring} \\
     & U & \mapsto & \KO_X(U)
\end{array} $$ is a sheaf
\end{prop}

\begin{proof}
It is a presheaf because if $V\subset U$ are regulous open sets, we have a restriction morphism 
$$ \begin{array}{ccc}
    \KO_X(U) & \to & \KO_X(V)  \\
    f & \mapsto & f_{\mid V}
\end{array}$$
In order to prove that it is a sheaf, we consider a regulous open set $U$ and an open cover $\{U_i\}_{i\in I}$ of $U$. By Lemma \ref{LemIrreducibleLocal}, we can suppose that $X$ is irreducible. Let $\{f_i\}_{i\in I}$ be such that $f_i\in\KO_X(U_i)$ for all $i\in I$ and such that for all $i,j\in I$
$$ (f_i)_{\mid U_i\cap U_j} = (f_j)_{\mid U_i\cap U_j} $$
Then we can define the continuous function 
$$ \begin{array}{ccccl}
    f :& U & \to & \C & \\
     & x & \mapsto & f(x) &\text{if }x\in U_i
\end{array}$$
Moreover, for all $i\in I$, there is a Z-open set $V_i\cap U_i$ which is  Z-dense for the induced topology on $U_i$, on which $f_i$ is regular. Since $X$ is irreducible, all the Z-open sets are Z-dense so one can take any $i_0\in I$ and get that $f$ is regular on the Z-open Z-dense set $V_{i_0}\cap U_{i_0}$.
\end{proof}

Thanks to Proposition \ref{PropLocalisation}, we get the following extension theorem for regulous functions defined on a principal open set of an affine variety.
\newpage
\begin{prop}\label{PropLoja}
Let $f\in \KO(X(\C))$ and let $g:\Dcal^0(f) \to \C$ be an element of $\KO_X(\Dcal^0(f))$. Then there exists $N\in \mathbb{N}$ such that the function 
$$f^Ng := \left\{ \begin{array}{cl}
    f^N(x)g(x) & \text{ if }x\in\Dcal^0(f)\\
    0 & \text{ else}
\end{array} \right.$$ is an element of $\KO(X(\C))$.
\end{prop}

\begin{proof}
Let $f\in\KO(X(\C))$ and $g:\Dcal^0(f) \to \C$. We write $\pi : X^+\to X$ the seminormalization morphism. Since $\pi^{-1}(\Dcal^0(f)) = \Dcal(f\circ\pi)$, we get $$g\circ\pi : \Dcal(f\circ\pi)\to \C \in \KO_{X^+}(\Dcal(f\circ\pi))$$
Moreover, by Theorem \ref{TheoRatioContEgalPolySN}, we have $f\circ\pi \in \C[X^+]$. Then $\Dcal(f\circ\pi)$ is affine and we can apply Proposition \ref{PropLocalisation} with $q:=f\circ\pi$. So we obtain an integer $N$ such that $(f\circ\pi)^N g\circ\pi\in\KO(X^+(\C)) = \C[X^+]$. Finally, thanks to Theorem \ref{TheoRatioContEgalPolySN} again, we get $$ \exists N\in\mathbb{N} \text{\hspace{0.8cm} } f^Ng\in\KO(X(\C))$$
\end{proof}

The previous proposition allows us to described the structure of the ring of regulous functions defined on a principal open set.
\begin{prop}\label{PropRegDefSurOuvertPrincipal}
Let $X$ be an affine variety and let $U := \Dcal^0(f)$ be a regulous open set with $f\in \KO(X(\C))$. Then, the restriction morphism from $\KO_X(X(\C))$ to $\KO_X(U)$ induces an isomorphism
$$ \KO_X(X(\C))_f \simeq \KO_X(U)$$
\end{prop}

\begin{proof}
Let $\psi : \KO_X(X(\C))\to \KO_X(U)$ be the restriction morphism. Since the restriction of $f$ to $U$ doesn't vanish, we get the induced morphism $$ \psi_f : \KO_X(X(\C))_f \to \KO_X(U) $$ 
Let $g\in \KO_X(U)$. Then, by Proposition \ref{PropLoja}, there exists $N\in \mathbb{N}$ such that $f^Ng \in \KO_X(X(\C))$. So $\psi_f$ is surjective.\\
Let $g/f^n\in \KO_X(X(\C))_f$ such that $\psi_f(g/f^n) = 0$. Then $(g/f^n)_{\mid U} = g_{\mid U}/f^n_{\mid U} = 0$ implies that $g$ vanish on the dense set $U$. So $g$ vanish on all $X(\C)$ because it is continuous for the euclidean topology and we get $g/f^n = 0$. Hence $\psi_f$ is injective.
\end{proof}

We now get the main result of this section which is a generalization of Theorem \ref{TheoRatioContEgalPolySN} for schemes.
\begin{theorem}\label{TheoIsoSchemaSN}
Let $X$ be an affine variety and $\pi : X^+ \to X$ be its seminormalization morphism. Then $(\pi,\pi^*)$ is an isomorphism of ringed spaces between $(X,\KO_X)$ and $(X^+,\Ocal_{X^+})$.
\end{theorem}

\begin{proof}
Let $f\in\KO(X(\C))$ and $g:\Dcal^0(f)\to \C \in \KO_X(\Dcal(f))$. By Proposition \ref{PropRegDefSurOuvertPrincipal}, we have $g\in\KO_X(X(\C))_f$, so there exists $h\in\KO(X(\C))$ and $N\in\mathbb{N}$ such that $g = h/f^N$ on $\Dcal^0(f)$. Then, by Theorem \ref{TheoRatioContEgalPolySN}, we get $g = h\circ\pi / f\circ\pi^N \in \Ocal_{X^+}(\Dcal(f\circ\pi))$.


Now, let $g : U\to \C \in \KO_X(U)$ with $U$ a regulous open set. Then we can write $U = \bigcup_{i=1}^s \Dcal(f_i)$ with $f_i\in\KO(X(\C))$. Then $g_{\mid\Dcal(f_i)}\in\KO(\Dcal(f_i))$, so $g\circ \pi_{\mid\Dcal(f_i\circ\pi)}\in\Ocal_{X^+}(\Dcal(f_i\circ\pi))$ and finally we get $g\circ\pi\in\Ocal_{X^+}(U)$ because $\Ocal_{X^+}$ is a sheaf.

Conversely, if $U^+$ is a Z-open set of $X^+(\C)$ and $f\in\Ocal_{X^+}(U^+)$, then $f\circ\pi^{-1} : U\to \C \in \KO_X(U)$ because $\pi$ is a bicontinuous birational morphism.
\end{proof}

\begin{rmq}
A part of the paper \cite{FHMM} is dedicated to prove Cartan's theorems A and B for real algebraic varieties with the sheaf of regulous functions. Since those theorems are true for complex algebraic varieties with the sheaf of regular functions (see \cite{Serre} Theorem 2 section 45 and Theorem 3 section 3), then Theorem \ref{TheoIsoSchemaSN} says that those results are also true for complex algebraic varieties with the sheaf of regulous functions.\\
\end{rmq}


\begin{cor}
Let $\p \subset \C[X^+]$ and $\q \in \C[X]$ be prime ideals such that $\p\cap \C[X] = \q$. Then 
$$ \Ocal_X(\Dcal(\q))^+ \simeq \KO_X(\Dcal(\q)) $$
\end{cor}

\begin{proof}
By Theorem \ref{TheoIsoSchemaSN} and since localization and seminormalization commute (see \cite{T}), we have 
$$ \Ocal_X(\Dcal(\q))^+ \simeq (\C[X]_{\q})^+ \simeq \C[X^+]_{\p} \simeq \Ocal_{X^+}(\Dcal(\p)) \simeq \KO_X(\Dcal(\q)) $$
\end{proof}

\bibliographystyle{abbrv}
\bibliography{Biblio.bib}

\begin{thebibliography}{10}

\bibitem{Andre}
A.~Andreotti and F.~Norguet.
\newblock La convexit\'{e} holomorphe dans l'espace analytique des cycles d'une
  vari\'{e}t\'{e} alg\'{e}brique.
\newblock {\em Ann. Scuola Norm. Sup. Pisa Cl. Sci. (3)}, 21:31--82, 1967.

\bibitem{AM}
M.~F. Atiyah and I.~G. Macdonald.
\newblock {\em Introduction to commutative algebra}.
\newblock Addison-Wesley Publishing Co., Reading, Mass.-London-Don Mills, Ont.,
  1969.

\bibitem{BFMQ}
F.~Bernard, G.~Fichou, J.-P. Monnier, and R.~Quarez.
\newblock Saturation, seminormalization and homeomorphisms of algebraic
  varieties.
\newblock {\em arXiv preprint arXiv:2203.09967}, 2022.

\bibitem{BDT}
A.~Bia{\l}o\.{z}yt, M.~P. Denkowski, and P.~Tworzewski.
\newblock On the {N}ullstellensatz for c-holomorphic functions with algebraic
  graphs.
\newblock {\em Ann. Polon. Math.}, 125(1):1--11, 2020.

\bibitem{BDTW}
A.~Bia{\l}o{\.z}yt, M.~P. Denkowski, P.~Tworzewski, and T.~Wawak.
\newblock On the growth exponent of c-holomorphic functions with algebraic
  graphs.
\newblock {\em arXiv preprint arXiv:1406.5293}, 2020.

\bibitem{Dav}
E.~D. Davis.
\newblock On the geometric interpretation of seminormality.
\newblock {\em Proc. Amer. Math. Soc.}, 68(1):1--5, 1978.

\bibitem{Eisen}
D.~Eisenbud.
\newblock {\em Commutative algebra: with a view toward algebraic geometry},
  volume 150.
\newblock Springer Science \& Business Media, 2013.

\bibitem{FHMM}
G.~Fichou, J.~Huisman, F.~Mangolte, and J.-P. Monnier.
\newblock Fonctions r\'{e}gulues.
\newblock {\em J. Reine Angew. Math.}, 718:103--151, 2016.

\bibitem{FMQ}
G.~Fichou, J.-P. Monnier, and R.~Quarez.
\newblock Weak- and semi-normalization in real algebraic geometry.
\newblock {\em Ann. Sc. Norm. Super. Pisa Cl. Sci. (5)}, 22(3):1511--1558,
  2021.

\bibitem{I}
S.~Iitaka.
\newblock {\em An introduction to birational geometry of algebraic varieties}.
\newblock Springer-Verlag, New York-Berlin, 1982.

\bibitem{Kollar2}
J.~Koll\'{a}r.
\newblock {\em Singularities of the minimal model program}, volume 200 of {\em
  Cambridge Tracts in Mathematics}.
\newblock Cambridge University Press, Cambridge, 2013.
\newblock With a collaboration of S\'{a}ndor Kov\'{a}cs.

\bibitem{Kollar3}
J.~Koll\'{a}r.
\newblock Variants of normality for {N}oetherian schemes.
\newblock {\em Pure Appl. Math. Q.}, 12(1):1--31, 2016.

\bibitem{Kollar}
J.~Koll\'{a}r and K.~Nowak.
\newblock Continuous rational functions on real and {$p$}-adic varieties.
\newblock {\em Math. Z.}, 279(1-2):85--97, 2015.

\bibitem{Kuch}
W.~Kucharz and K.~Kurdyka.
\newblock From continuous rational to regulous functions.
\newblock In {\em Proceedings of the {I}nternational {C}ongress of
  {M}athematicians---{R}io de {J}aneiro 2018. {V}ol. {II}. {I}nvited lectures},
  pages 719--747. World Sci. Publ., Hackensack, NJ, 2018.

\bibitem{LV}
J.~V. Leahy and M.~A. Vitulli.
\newblock Seminormal rings and weakly normal varieties.
\newblock {\em Nagoya Math. J.}, 82:27--56, 1981.

\bibitem{Multicross}
J.~V. Leahy and M.~A. Vitulli.
\newblock Weakly normal varieties: the multicross singularity and some
  vanishing theorems on local cohomology.
\newblock {\em Nagoya Math. J.}, 83:137--152, 1981.

\bibitem{Central}
J.-P. Monnier.
\newblock Central algebraic geometry and seminormality.
\newblock {\em arXiv preprint arXiv:2103.09610}, 2021.

\bibitem{Serre}
J.-P. Serre.
\newblock Faisceaux alg\'{e}briques coh\'{e}rents.
\newblock {\em Ann. of Math. (2)}, 61:197--278, 1955.

\bibitem{Shafa}
I.~R. Shafarevich.
\newblock {\em Basic algebraic geometry. 1}.
\newblock Springer, Heidelberg, second edition, 1986.
\newblock Varieties in projective space.

\bibitem{Shafa2}
I.~R. Shafarevich.
\newblock {\em Basic algebraic geometry. 2}.
\newblock Springer, Heidelberg, second edition, 1987.
\newblock Schemes and complex manifolds.

\bibitem{Swan}
R.~G. Swan.
\newblock On seminormality.
\newblock {\em J. Algebra}, 67(1):210--229, 1980.

\bibitem{T}
C.~Traverso.
\newblock Seminormality and {P}icard group.
\newblock {\em Ann. Scuola Norm. Sup. Pisa Cl. Sci. (3)}, 24:585--595, 1970.

\bibitem{V2}
M.~A. Vitulli.
\newblock Corrections to: ``{S}eminormal rings and weakly normal varieties''
  [{N}agoya {M}ath. {J}. {\bf 82} (1981), 27--56; {MR}0618807 (83a:14015)] by
  {J}. {V}. {L}eahy and {V}itulli.
\newblock {\em Nagoya Math. J.}, 107:147--157, 1987.

\bibitem{V}
M.~A. Vitulli.
\newblock Weak normality and seminormality.
\newblock In {\em Commutative algebra---{N}oetherian and non-{N}oetherian
  perspectives}, pages 441--480. Springer, New York, 2011.

\bibitem{Whitney}
H.~Whitney.
\newblock {\em Complex analytic varieties}.
\newblock Addison-Wesley Publishing Co., Reading, Mass.-London-Don Mills, Ont.,
  1972.

\end{thebibliography}

François Bernard, Université d’Angers, LAREMA, UMR 6093 CNRS, Faculté des Sciences Bâtiment I, 2 Boulevard Lavoisier, F-49045 Angers cedex 01, France\\
\textit{E-mail address: \textbf{bernard@math.univ-angers.fr}}

\end{document}